\documentclass[12pt,reqno]{amsart}
\usepackage{amssymb,latexsym,amscd, verbatim} 
\usepackage{amsmath}
\makeatletter
\renewcommand{\pod}[1]{\allowbreak\mathchoice
  {\if@display \mkern 18mu\else \mkern 8mu\fi (#1)}
  {\if@display \mkern 18mu\else \mkern 8mu\fi (#1)}
  {\mkern4mu(#1)}
  {\mkern4mu(#1)}
}

\usepackage{enumerate}[1.)]
\usepackage{pb-diagram}  
\renewcommand{\eqref}[1]{(\ref{#1})}   
\theoremstyle{plain}
\newtheorem{theorem}{Theorem}[section]

\usepackage{geometry}

\newtheorem{corollary}{Corollary}[section]

\newtheorem{lemma}{Lemma}[section]

\newtheorem{proposition}{Proposition}[section]
\newtheorem{algorithm}{Algorithm}

\newtheorem{definition}{Definition}

\theoremstyle{remark}

\begin{document}

\title{Level of distribution of unbalanced convolutions}
\date{\today}
\author{\'Etienne Fouvry}
\address{ Univ. Paris--Sud, Institut de Math\' ematique d'Orsay, UMR 8628, Orsay, F--91405 France, CNRS, Orsay, F--91405, France}
\email{Etienne.Fouvry@u-psud.fr}
\author{Maksym Radziwi\l\l}
\address{Department of Mathematics \\
  Caltech \\
  1200 E California Blvd
  Pasadena, CA, 91125, USA  }
\email{maksym.radziwill@gmail.com}

\begin{abstract} We show that if an essentially arbitrary sequence supported on an interval containing $x$ integers, is convolved with a tiny Siegel-Walfisz-type sequence supported on an interval containing $\exp((\log x)^{\varepsilon})$ integers then the resulting multiplicative convolution has (in a weak sense) level of distribution $x^{1/2 + 1/66 - \varepsilon}$ as $x$ goes to infinity. This dispersion estimate has a number of consequences for: the distribution of the $k$th divisor function to moduli $x^{1/2 + 1/66 - \varepsilon}$ for any integer $k \geq 1$, the distribution of products of exactly two primes in arithmetic progressions to large moduli, the distribution of sieve weights of level $x^{1/2 + 1/66 - \varepsilon}$ to moduli as large as $x^{1 - \varepsilon}$ and for the Brun-Titchmarsh theorem for almost all moduli $q$ of size $x^{1 - \varepsilon}$, lowering the long-standing constant $4$ in that range. Our result improves and is inspired by earlier work of Green (and subsequent work of Granville-Shao) which is concerned with the distribution of $1$-bounded multiplicative functions in arithmetic progressions to large moduli. As in these previous works the main technical ingredient are the recent estimates of Bettin-Chandee for trilinear forms in Kloosterman fractions and the estimates of Duke-Friedlander-Iwaniec for bilinear forms in Kloosterman fractions.  

   \end{abstract}

\maketitle

\renewcommand{\theenumi}{(\roman{enumi})}



 
\section{Introduction}

 A major open problem in prime number theory is to show  the existence of some $\delta >0$ such that for any integer $a \neq 0$ and for any $A>0$ we have   
 \begin{equation} \label{eq:goal}
\sum_{\substack{q \leq x^{1/2 + \delta}\\ (q,a) = 1}} \Big | \sum_{\substack{p \leq x \\ p \equiv a \pmod{q}}} 1 - \frac{1}{\varphi(q)} \sum_{\substack{p \leq x \\ (p,q) = 1}} 1 \Big | \ll_{a,A} x (\log x)^{-A}
\end{equation}
uniformly for $x\geq 2$.
We would then say that \textit{primes have a level of distribution} $x^{1/2 + \delta}$ \textit{in a weak sense}, and call $\tfrac{1}{2} + \delta$ an \textit{exponent of distribution of the primes in a weak sense}. If we could establish a similar statement but with the maximum over $(a,q) = 1$ inside the sum over $q$ we would then drop the \textit{weak} adjective (see \cite{Fo-Ma} for a precise definition). For brevity we will not further distinguish between the two terms since we will be only concerned with the former ``weak sense''

By Chebyschev's inequality \eqref{eq:goal} implies that for ``almost all'' (i.e all with the exception of a density zero subset) moduli $q \leq x^{1/2 + \delta}$ the primes are well-distributed in arithmetic progressions $n \equiv a \pmod{q}$. 

The inequality \eqref{eq:goal} follows for $\delta < 0$ from the Bombieri-Vinogradov theorem, and for $\delta = 0$ and $A < 2$ from work of Bombieri-Friedlander-Iwaniec \cite{B-F-I2}. Zhang \cite{Zhang} established \eqref{eq:goal} for some $\delta > 0$ with $q$ restricted to $x^{\varepsilon}$ smooth moduli (see also \cite{Polymath}). The problem of establishing \eqref{eq:goal} for some $\delta > 0$ is challenging since it lies beyond the capability of the Generalized Riemann Hypothesis. 

Underpinning any current approach to \eqref{eq:goal} are \textit{dispersion estimates} originally invented by Linnik. Roughly a dispersion estimate asserts that for $M,N \geq 1$ and two arbitrary sequences $\boldsymbol \alpha = (\alpha_{m})_{M < m \leq 2M}$ and $\boldsymbol \beta = (\beta_{n})_{N < n \leq 2N}$ of complex numbers satisfying some minor technical conditions, we have for any $a \neq 0$ fixed, 
\begin{equation} \label{eq:goal2}
\sum_{\substack{Q \leq q \leq 2Q \\ (q,a) = 1}} \Big | \sum_{\substack{x < m n \leq 2 x \\ m n \equiv a \pmod{q}}} \beta_{m} \gamma_{n} - \frac{1}{\varphi(q)} \sum_{\substack{x < m n \leq 2 x \\ (m n, q) = 1 }} \beta_{m} \gamma_{n} \Big | \ll_{a,A} x (\log x)^{-A} 
\end{equation}
for $x \asymp M N$, uniformly in $Q \leq x^{1/2 + \delta}$ for some $\delta > 0$. As usual the case of $\delta < 0$ falls within the scope of techniques related to the Bombieri-Vinogradov theorem, and is well-understood (see \cite[Theorem 0]{BFI1} or \cite[Theorem 9.16]{Opera}).

A necessary assumption in a dispersion estimate is that at least one of the sequences is well-distributed in arithmetic progressions having small moduli. This is referred to as a \textit{Siegel-Walfisz condition}. 
\begin{definition}
  We say that a sequence $\boldsymbol \beta = (\beta_n)$ of complex numbers satisfies a Siegel-Walfisz condition (alternatively we also say that $\boldsymbol \beta$ is \textit{Siegel-Walfisz}), if there exists an integer $k > 0$ such that for any fixed $A > 0$, uniformly in $x \geq 2$, $q > |a| \geq 1, r \geq 1$ and $(a,q) = 1$, we have,
  \begin{equation} \label{cond}
  \sum_{\substack{x < n \leq 2x \\ n \equiv a \pmod{q} \\ (n,r) = 1}} \beta_n = \frac{1}{\varphi(q)} \sum_{\substack{x < n \leq 2x \\ (n, q r) = 1}} \beta_{n} + O_{A}(\tau_k(r) \cdot x (\log x)^{-A}). 
  \end{equation}
  where $\tau_{k}(n)$ is the $k$-th divisor function $\tau_k(n) := \sum_{n_1 \ldots n_k = n} 1$. 
\end{definition}

For $\delta > 0$ there are few results that address \eqref{eq:goal} in wide generality. As we already mentioned at least one of the sequence $\boldsymbol \alpha$, $\boldsymbol \beta$ needs to be Siegel-Walfisz. In all the cases that are known (i.e \cite[Theorem 3]{BFI1}, \cite[Th\'eor\`eme 1]{FoActaMath} and \cite[Corollaire 1]{FoAnnENS}) the Siegel-Walfisz sequence needs to be supported on an interval of length at least $x^{\varepsilon} \cdot (Q / \sqrt{x} + 1)^{2}$ (and no longer than say $x^{1/6 - \varepsilon}$ or $x^{1/12 - \varepsilon}$). In particular the length of this interval is at least a power of $x$ as soon as $Q$ increases beyond $\sqrt{x}$ by a small power of $x$. 


Our first result is a new dispersion estimate that roughly shows that \eqref{eq:goal2} can be obtained with $Q = x^{1/2 + 1/66 - \varepsilon}$ even if the Siegel-Walfisz sequence $\boldsymbol \beta$ is supported on a tiny interval of length $\exp((\log x)^{\varepsilon})$ for any sufficiently small $\varepsilon > 0$. We find this rather striking, since this means that a tiny smoothing of an otherwise arbitrary sequence supported on $x$ integers allows one to suddenly reach a level of distribution $x^{1/2 + 1/66 - \varepsilon}$. We call such a convolution of two sequences of drastically different sizes an \textit{unbalanced convolution}.  



\begin{corollary} \label{cor:Main}
  Let $k > 0$ and $\varepsilon > 0$ be given. Let $\boldsymbol \alpha = (\alpha_{m})_{M < m \leq 2M}$ and $\boldsymbol \beta = (\beta_{n})_{N < n \leq 2N}$ be two sequences of complex numbers such that $|\alpha_{m}| \leq \tau_k(m)$ and $|\beta_{n}| \leq \tau_k(n)$ for all $m,n \geq 1$. Suppose that $\boldsymbol \beta$ is Siegel-Walfisz. Then for every $A > 0$, uniformly in $M,N \geq 2$ with $M N / 2 \leq x \leq 4 M N$ we have, 
  \begin{equation} \label{eq:dispmain}
  \sum_{\substack{Q \leq q \leq 2Q \\ (q,a) = 1}} \Big | \sum_{\substack{x < m n \leq 2 x \\ m n \equiv a \pmod{q}}} \alpha_{m} \beta_n - \frac{1}{\varphi(q)} \sum_{\substack{x < m n \leq 2 x \\ (m n, q) = 1}} \alpha_m \beta_{n} \Big | \ll_{A} x (\log x)^{-A}
  \end{equation}
  provided that either of the following three conditions holds
  \begin{enumerate}
  \item $\exp((\log x)^{\varepsilon}) \leq N \leq Q^{- 11 / 12} \cdot x^{17/36 - \varepsilon}$ and $1 \leq |a| \leq x /12$.
  \item $\exp((\log x)^{\varepsilon}) \leq N \leq x^{7/90 - \varepsilon}$, $Q \leq x^{53/105 - \varepsilon}$ and $1 \leq |a| \leq x/12$.
  \item $\exp((\log x)^{\varepsilon}) \leq N \leq x^{101/630 - \varepsilon}$, $Q \leq x^{53/105 - \varepsilon}$ and $1 \leq |a| \leq (x/4)^{\varepsilon / 1000}$. 
    \end{enumerate}
  \end{corollary}

Note that the left-hand side of \eqref{eq:dispmain} is identically zero if $x$ falls outside of the interval $[M N / 2, 4 M N]$. Here $i)$ gives the strongest estimate in the $Q$-aspect for very small $N$, allowing for $Q$ to go up to $x^{1/2 + 1/66 - 3 \varepsilon}$ provided that $N \leq x^{\varepsilon}$, where-as $ii)$ and $iii)$ give stronger uniformity in the $N$-aspect at the price of a slightly weaker level of distribution. 

A numerically stronger, but conditional, version of Corollary \ref{cor:Main} appears in Fouvry's thesis \cite{FoTh}. Fouvry's result depends on the assumption of the still unproven Hooley's $R^{\star}$--conjecture on cancellations in short incomplete Kloosterman sums. To obtain the unconditional Corollary \ref{cor:Main} we appeal instead to results of Duke-Friedlander-Iwaniec \cite{D-F-I} and Bettin-Chandee \cite{Be-Ch}. These results can be used as unconditional substitutes for Hooley's $R^{\star}$--conjecture ``on average''. A similar observation is implicit in the recent work of Green \cite{Green} which is the second starting point for our work.


  %

Our dispersion estimate has a number of interesting corollaries, many of them relying on the observation that most integers $n$ can be factored as $n = p m$ with $p$ a small prime in the range $[\exp((\log x)^{\varepsilon}), x^{\varepsilon}]$. The first corollary concerns the distribution of the $k$th divisor function in arithmetic progressions to large moduli.

\begin{corollary}\label{dk} Fix an integer  $k\geq 1$ and $\varepsilon >0$. Then uniformly for $x \geq 2$, $Q\leq x^{\frac{17}{33} - \varepsilon}$ and $1 \leq \vert a \vert \leq x/12$, one has the inequality
\begin{equation} \label{eq:dk}
\sum_{\substack{Q \leq q \leq 2Q \\ (q,a)=1}} \Big | \sum_{\substack{x < n \leq 2x \\ n \equiv a \pmod{q}}} \tau_k(n) - \frac{1}{\varphi(q)} \sum_{\substack{x < n \leq 2x \\ (n,q) = 1}} \tau_k(n) \Big | \ll_{\varepsilon} \frac{x}{(\log x)^{1-\varepsilon}}.
\end{equation}
\end{corollary}
An interesting aspect of Corollary \ref{dk} is that it becomes stronger as $k$ increases. The trivial bound for the left-hand side of \eqref{eq:dk} is $\ll x (\log x)^{k - 1}$. Therefore when $k$ is large we are saving $k$ powers of the logarithm over the trivial bound. We note that stronger results are known in the cases $k = 1,2,3$ (see \cite{d3}). However Corollary \ref{dk} is the first non-trivial result for $k > 3$ in the range $Q > x^{1/2 + \delta}$ with $\delta > 0$. It is likely that if we could replace $(\log x)^{1 - \varepsilon}$ with $(\log x)^{1 + \varepsilon}$ in \eqref{eq:dk} then interesting consequences for prime numbers would ensue. 

Corollary \ref{dk} follows from a result that applies to general multiplicative functions.

\begin{corollary} \label{cor:mult} Fix an integer $k \geq 1$ and $\varepsilon > 0$. Let $g: \mathbb{N} \rightarrow \mathbb{C}$ be a multiplicative function such that $|g(n)| \leq \tau_{k}(n)$ for all integer $n \geq 1$. Suppose that the sequence $n \mapsto \mathbf{1}_{n \text{ is prime}} \cdot g(n)$ is Siegel-Walfisz. Then, uniformly for $x \geq 2$, $Q \leq x^{\frac{17}{33} - \varepsilon}$, and $1 \leq |a| \leq x/12$, we have
  \begin{equation} \label{eq:cormult}
  \sum_{\substack{Q < q \leq 2Q \\ (a,q) = 1}} \Big | \sum_{\substack{x < n \leq 2x \\ n \equiv a \pmod{q}}} g(n) - \frac{1}{\varphi(q)} \sum_{\substack{x < n \leq 2x \\ (n,q) = 1}} g(n) \Big | \ll_{\varepsilon} \frac{x}{(\log x)^{1 - \varepsilon}}.
  \end{equation}
\end{corollary}

Corollary \ref{cor:mult} improves on work of Green \cite{Green}, and subsequent work of Granville-Shao \cite{GrSh} that extended Green's work to all moduli. We notice that if one is only interested in prime moduli then the assumption that $g(p)$ is Siegel-Walfisz can be omitted (see Theorem \ref{NOS-W}). Both Green and Granville-Shao restrict their attention to multiplicative functions $g$ such that $|g(n)| \leq 1$ and obtain in these cases a weaker exponent of distribution $\tfrac{20}{39} < \tfrac{17}{33}$. Roughly speaking Corollary \ref{cor:mult} is non-trivial for multiplicative functions for which there exists an $\varepsilon > 0$ such that $|g(p)| > \varepsilon$ for a positive proportion of primes. We are not aware of any naturally occurring multiplicative function that does not fulfill this condition.




Since most integers with exactly $k$ prime factors have a small prime factor in the range $[\exp((\log x)^{\varepsilon}), x^{\varepsilon}]$ we can apply our dispersion estimate in this case as-well. 

\begin{corollary} \label{corollary4} Fix $k \geq 2$ integer and $\varepsilon > 0$. 
  Let $\Omega(n)$ denote the number of prime factors of $n$ counted with multiplicity. 
  Then, uniformly for $x \geq 12$, $Q \leq x^{17/33 - \varepsilon}$ and $1 \leq |a| \leq x/12$, we have
\begin{equation} \label{eq:corollary4}
\sum_{\substack{Q \leq q \leq 2Q \\ (a,q) = 1}} \Big | \sum_{\substack{x < n \leq 2x \\ n \equiv a \pmod{q} \\ \Omega(n) = k}} 1 - \frac{1}{\varphi(q)} \sum_{\substack{x < n \leq 2x \\ (n,q) = 1 \\ \Omega(n) = k}} 1 \Big | = o \Big ( x \cdot \frac{(\log\log x)^{k - 1}}{\log x} \Big )
\end{equation}
as $x$ tends to infinity.
\end{corollary}

A sieve bound shows that the number of integers with exactly $k$ prime factors congruent to $a \pmod{q}$ with $(a,q) = 1$ is $$\ll_{k} \frac{x}{\varphi(q)} \cdot \frac{(\log\log x)^{k - 1}}{\log x}.$$
Thus Corollary \ref{corollary4} implies that as $x \rightarrow \infty$, for almost all $q \leq x^{17/33 - \varepsilon}$ with $(q,a)=1$, we have
$$
\sum_{\substack{x < n \leq 2 x \\ n \equiv a \pmod{q} \\ \Omega(n) = k}} 1 = (1 + o(1)) \cdot \frac{1}{\varphi(q)} \sum_{\substack{x < n \leq 2x \\ (n,q) = 1 \\ \Omega(n) = k}} 1. 
$$
We are unable at the moment to address the case of $k = 1$ which remains a challenging open problem. Previous results, such as the results of Fouvry-Iwaniec \cite{FoIwMathematika} required $n = p_1 p_2 \ldots p_k$ to factor in a specific way where each $p_i$ is localized in certain special intervals. Thus this did not allow one to obtain a result like Corollary \ref{corollary4}. We notice, by the way, that a result saying that $n = p_1 p_2$ with $p_1 \asymp \exp((\log x)^{\varepsilon})$ has an exponent of distribution of at least $\tfrac{17}{33} - \varepsilon$ follows immediately from Corollary \ref{cor:Main}. Although we already said it before, we do find it striking that such a small perturbation of the sequence of primes leads to such a high level of distribution.

Combining Corollary \ref{cor:Main} and Dirichlet's divisor switching technique allows us to achieve a very high level of distribution for sieve weights of level up to $x^{17/33 - \varepsilon}$. 


\begin{corollary} \label{cor:dirichlet} Let $\varepsilon > 0$ and $k > 0$ be given. Let $\boldsymbol \lambda = (\lambda_d)_{1 \leq d \leq z}$ be a sequence of complex numbers with $|\lambda_d| \leq \tau_k(d)$. Then, we have
$$
\sum_{\substack{q \sim Q  \\ (q,a)=1}} \Big | \sum_{\substack{x < n \leq 2 x \\ n \equiv a \pmod{q}}} \Big ( \sum_{\substack{d | n \\ d \leq z}} \lambda_d \Big ) - \frac{1}{\varphi(q)} \sum_{\substack{x < n \leq 2x \\ (n,q) = 1}} \Big ( \sum_{\substack{d | n \\ d \leq z}} \lambda_d \Big ) \Big | \ll _{A} \frac{x}{(\log x)^{A}},
$$
provided that either of the following three conditions holds: 
\begin{enumerate}
\item $x \geq 12$, $z \leq x^{53/105 - \varepsilon}$,  $x^{1 - \varepsilon} > Q > x^{529/630 + \varepsilon}$ and $1 \leq |a| \leq x^{\varepsilon / 10000}$
\item $x \geq 12$, $z \leq x^{53/105 - \varepsilon}$, $x^{1 - \varepsilon} > Q > x^{83 / 90 + \varepsilon}$ and $1 \leq |a| \leq x^{1 - 3 \varepsilon}$
\item $x \geq 12$, $z \leq x^{1/2 + \delta - \varepsilon}$, $x^{1 - \varepsilon} > Q > x^{(71 + 66 \delta)/72 + \varepsilon}$, and $1 \leq |a| \leq x^{1 - 3 \varepsilon}$ for any fixed $0 < \delta < \tfrac{1}{66}$. In this case the implicit constant in $\ll_{A}$ depends additionally on $\delta$. 
\end{enumerate}
\end{corollary}
To make comparison between the $Q$-ranges in Corollary \ref{cor:dirichlet} easier we record the numerical values
$$
\frac{529}{630} = 0.83962 \ldots \ , \ \frac{83}{90} = 0.92222 \ldots \ , \ \frac{71}{72} = 0.98611 \ldots
$$
Corollary \ref{cor:dirichlet} implies an improvement in the Brun-Titchmarsh theorem for almost all moduli $q \in [x^{1 - \varepsilon}, 2 x^{1 - \varepsilon}]$ which is new for all sufficiently small $\varepsilon > 0$.

\begin{corollary} \label{cor:tit}
  Let $\tfrac 12 < \theta < 1$, $\varepsilon > 0$ and $C > 0$ be given.
  Then, for every $A > 0$ we have,  
  \begin{equation} \label{eq:tit} 
  \#\Big \{ q \in [x^{\theta}, 2 x^{\theta}] \text{ and } (q,a) = 1: \pi(x; q,a) > \frac{(4 - \tfrac{2}{53} + \varepsilon) x}{\varphi(q) \log x} \Big \} \ll_{A} \frac{x^{\theta}}{(\log x)^{A}} 
  \end{equation}
  uniformly in $x \geq 2$ and $1 \leq |a| \leq (\log x)^{C}$. Moreover,
  \begin{enumerate}
  \item If $\theta > \frac{83}{90}$ then \eqref{eq:tit} holds uniformly in $x \geq 2$ and $1 \leq |a| \leq x^{1 - 3 \varepsilon}$
  \item If $\theta = 1 - \eta$ with $0 < \eta < \tfrac{1}{210}$ then \eqref{eq:tit} holds with $4 - \tfrac{2}{53}$ replaced by $\frac{66}{17 - 36 \eta}$, uniformly in $x \geq 2$ and $1 \leq |a| \leq x^{1 - 3\varepsilon}$
  \end{enumerate}
\end{corollary}.

Thus as $\theta$ approaches $1$ the constant in the Brun-Titchmarsh inequality approaches $\frac{66}{17} = 3.88 \ldots$ for almost all moduli $q \in [x^{\theta}, 2 x^{\theta}]$. The important point in Corollary \ref{cor:tit} is that it breaches in all ranges the value $4$ which is a consequence of techniques related to the Bombieri-Vinogradov theorem.  Conjecturally we expect that the optimal constant is equal to $1$.

For comparison we note that for $\frac{9}{10} < \theta < 1$ Fouvry showed in \cite{FouvryInv} that 
\begin{equation} \label{eq:fouvtit}
\pi(x; q, a) \leq \Big ( \frac{4}{2 - \theta} + o(1) \Big )\cdot \frac{x}{\varphi(q) \log x} 
\end{equation}
for almost all $q \in [x^{\theta}, 2 x^{\theta}]$ with $(q,a) = 1$, with at most $\ll_{A} x^{\theta} (\log x)^{-A}$ exceptions. 
In the range $\tfrac 12 < \theta < \tfrac {9}{10}$ Fouvry obtains numerically stronger results than \eqref{eq:fouvtit}. In the range $\tfrac 12 < \theta < 0.56$ the strongest currently known results are due to Baker-Harman \cite{BakerHarman}. Their work is motivated by applications to the size of the greatest prime factor of $p - 1$ with $p$ prime.


\subsection{Precise statement of the main theorem}

We now discuss the precise estimates that we obtain and from which Corollary \ref{cor:Main} follows. We also set-up here the notations that will be used throughout the remainder of the paper, and discuss the main ingredients in the proof of Corollary \ref{cor:Main}. 

Let $\boldsymbol \beta = (\beta_n)$ be a sequence of real numbers supported on $N < n \leq 2N$. The $\ell_{2}$ norm is defined by
$$
\| \boldsymbol \beta \|_{2, N} = \Big ( \sum_{N < n \leq 2N} |\beta_{n}|^2 \Big )^{1/2}. 
$$
To avoid constantly writing $N < n \leq 2N$ in subscripts, we will abbreviate this in subscripts as $n \sim N$, however outside of subscripts the notation $\sim$ corresponds to the usual asymptotic notation. 

Given a sequence $\boldsymbol \beta = (\beta_{n})$ and integers $a,q$ with $q \geq 1$ and $(a,q) = 1$ we measure the distribution of $\boldsymbol \beta$ in arithmetic progressions $a \pmod{q}$ by considering the discrepancy,
\begin{equation} \label{defE}
E ( \boldsymbol \beta, N, q, a) := \sum_{\substack{n\sim N \\ n\equiv a \pmod q}} \beta_n -\frac{1}{\varphi (q) } \sum_{\substack{n\sim N \\ (n,q) =1}}\beta_n,
\end{equation}
and the slight variant,
\begin{equation} \label{defEmodif}
E^{\star} ( \boldsymbol \beta, N, q, a; r) := \sum_{\substack{n \sim N \\ n \equiv a \pmod{q} \\ (n,r) = 1}} \beta_{n} - \frac{1}{\varphi(q)} \sum_{\substack{n \sim N \\ (n, q r) = 1}} \beta_{n}.
\end{equation}
defined for all $r \geq 1$.

When dealing with two sequences $\boldsymbol \alpha = (\alpha_{m})$ and $\boldsymbol \beta = (\beta_{n})$ supported respectively on integers $M < m \leq 2M$ and $N < n \leq 2N$, we define
$$
E ( \boldsymbol \alpha, \boldsymbol \beta, M, N, q, a) :=  \underset{\substack{m\sim M,\ n \sim N \\ mn \equiv a \pmod q}}{\sum \ \sum} \alpha_m \beta_n -\frac{1}{\varphi (q) }\underset{\substack{m\sim M,\ n \sim N \\ (mn ,q)=1}}{\sum \ \sum} \alpha_m \beta_n.
$$
We are interested in understanding 
\begin{equation}\label{defDelta}
\Delta (\boldsymbol \alpha, \boldsymbol \beta, M, N, Q, a)
:=
\sum_{\substack{q\sim Q \\ (q,a)=1}} \Bigl\vert E(\boldsymbol \alpha, \boldsymbol \beta, M, N , q, a)\Bigr\vert.
\end{equation}
The contribution of the small moduli to $\Delta(\boldsymbol \alpha, \boldsymbol \beta, M, N, q, a)$ is captured by
\begin{equation} \label{defcalE}
\mathcal{E}^{\star}( \boldsymbol \beta, N, Q) := \sum_{\delta} \sum_{v \sim Q / \delta} \sum_{(\delta', \delta) = 1} |E^{\star} ( \boldsymbol \beta, N, \delta, \delta'; v)|^2. 
\end{equation}
Notice that if the sequence $\boldsymbol \beta$ is Siegel-Walfisz, satisfies the bound $|\beta_{n}| \leq \tau_k(n)$ for some $k > 0$, and
if  $N > Q^{\varepsilon}$, then (according to Lemma \ref{874}) we have
$$
\mathcal{E}^{\star}(\boldsymbol \beta, N, Q) \ll_{A} N^2 Q (\log N)^{-A}.
$$

\begin{theorem}\label{Generalresult} Let $k \geq 1$ be an integer and let $\varepsilon > 0$ be given. Suppose that $M,N,Q$ and $D$ are such that $$M > Q (M N)^{\varepsilon} \ , \ M > N > D^{10}.$$
  Let $X = M N$. Let $\boldsymbol \alpha = (\alpha_{m})_{m \sim M}$ and $\boldsymbol \beta = (\beta_{n})_{n \sim N}$ be two sequences of complex numbers such that $|\alpha_{m}| \leq \tau_k(m)$ and $|\beta_{n}| \leq \tau_k(n)$ for all $m, n \geq 1$. Then, for all integers $1 \leq |a| \leq X / 3$ we have, 
  \begin{align*}
    \Delta (\boldsymbol \alpha, & \boldsymbol \beta, M, N, Q, a)  \ll_{k,\varepsilon} \| \boldsymbol \alpha \|_{2,M} \cdot \Big ( M Q^{-1} \mathcal{E}^{\star}(\boldsymbol \beta , N, Q)  \\ & +  
    (\log X)^{\kappa} N^2 Q  + (\log X)^{\kappa} D^{-\frac{1}{2}} MN^2 +
D^{C} X^\varepsilon ( M^\frac{3}{20} N^\frac{59}{20}  Q^\frac{33}{20}+ N^\frac{23}{8} Q^\frac{15}{8})  
\Big )^\frac{1}{2} 
  \end{align*}
for some constants $\kappa = \kappa(k)$ and $C = C(k, \varepsilon)$, depending only on $k$ and $k, \varepsilon$ respectively.   
\end{theorem}

Theorem \ref{Generalresult} is particularly useful when $Q$ is significantly larger than $\sqrt{M N}$ but $N$ is small compared to $M$ (for instance $M^{\varepsilon} < N < M^{1/105}$) When $N$ is not small compared to $M$ (say $N > M^{1/20}$ or $N > M^{1/10}$) we can appeal to previous results of Fouvry \cite[Th\'eor\`eme 1]{FoActaMath}, \cite[Corollaire 1]{FoAnnENS} and Bombieri-Friedlander-Iwaniec \cite[Theorem 3]{BFI1}. This leads to the stronger exponents appearing in $ii)$ and $iii)$ of Corollary \ref{cor:Main}. Finally Corollary \ref{cor:Main} also differs from Theorem \ref{Generalresult} by the addition of the multiplicative constraint $x < m n \leq 2 x$. This is essentially a technicality.

In \S \ref{NoSiegel} we will prove a variant of Theorem \ref{Generalresult}, where we make no Siegel--Walfisz type assumption for the sequence $\boldsymbol \beta$, but where the summation over $q$  is restricted to prime moduli.
\begin{theorem}\label{NOS-W}
  Let $k \geq 1$ be an integer and let $\varepsilon > 0$ be given. 
  Let $\boldsymbol \alpha = (\alpha_{m})_{m \sim M}$ and $\boldsymbol \beta = (\beta_{n})_{n \sim N}$ be two sequences of complex numbers such that $|\alpha_{m}| \leq \tau_k(m)$ and $|\beta_{n}| \leq \tau_k(n)$ for all $m,n \geq 1$. Let $X = M N$. Then, uniformly in $M, N \geq 1$, and $1 \leq |a| \leq X / 3$, 
 \begin{equation} \label{eq:ineqprimes} 
\sum_{\substack{q \sim Q \\ q \text{ prime } \\ (q,a)=1}} \Big \vert \, E(\boldsymbol \alpha, \boldsymbol \beta, M, N , q, a)\, \Big \vert  =O_{\varepsilon, k}
\Big (\, MN \exp \Big (-(\sqrt {\log N})/2 \Big ) \, \Big ), 
\end{equation}
provided that $\exp((\log X)^{\varepsilon}) \leq N \leq Q^{-11/12} X^{17/36 - \varepsilon}$ and $Q \geq \exp(\sqrt{\log N})$.
\end{theorem}

In particular \eqref{eq:ineqprimes} is true for $Q = X^{1/2}$ and $\exp((\log X)^{\varepsilon}) \leq N \leq X^{1/72 - \varepsilon}$.
One can formulate the analogue of Corollary \ref{cor:Main} for Theorem \ref{NOS-W}, but since the details are very similar we forego this. 

Before we embark on the proof of Theorem \ref{Generalresult} and \ref{NOS-W} let us quickly explain where lies, on the technical level, the main difference compared to earlier dispersion estimates. As we already remarked several times in the introduction previous dispersion estimates required one to take $N > X^{2 \delta}$ if one was aiming to achieve a level of distribution $Q = X^{1/2 + \delta - \varepsilon}$ with $\delta > 0$ small. 

Our proof starts in the usual way, by applying the dispersion method of Linnik, and after several transformations, the main problem boils down to giving a non-trivial bound for the following trilinear sum 
 $$
\Sigma (U,V,W):= \frac{1}{U} \sum_{u\sim U}\  \sum_{v\sim V}\  \sum_{w\sim W} x_u y_v z_w e \Bigl( u \frac{\overline v}{w}\Bigr),
 $$
 where $x_u$, $y_v$ and $z_w$ are unspecified coefficients, with modulus less than $1$. 

 Let us explain in more detail: In the situation in which $N = M^{o(1)}$ we roughly find that $Q^{1 - o(1)} \leq V, W \leq Q^{1 + o(1)}$ and $1 \leq U \leq Q^{2 + o(1)} / M$, while the bound that we are looking for is $M N^2 X^{-\varepsilon}$. Thus we notice that the trivial bound is $Q^{2 + o(1)}$ while any non-trivial bound of the form $Q^{2 - \delta}$ for some fixed $\delta > 0$ would be sufficient to establish an exponent of distribution $\frac{1}{2 - \delta} > \tfrac{1}{2}$. 

 If one were to follow through the proof of previous dispersion estimates we would apply here Cauchy-Schwarz on the $v$ variable in order to smoothen it. This then leads after Poisson summation to the problem of bounding $Q^{o(1)}$ Kloosterman sums of modulus $Q^{2 + o(1)}$. Subsequently applying the Weil bound fails to recover the trivial bound! Another option would be to apply spectral theory and this remains an interesting possibility that deserves to be explored further. However one can circumvent these difficulties by using the results of Bettin-Chandee \cite{Be-Ch} and Duke-Friedlander-Iwaniec \cite{D-F-I}. Those rely on a variant of the amplification method which is particularly efficient in the (most difficult) regime in which $V$ and $W$ are nearby.  

 However when $N$ is not tiny it is a useful idea to apply Cauchy-Schwarz followed by spectral theory. This is the approach taken in earlier dispersion estimates. In this case the application of Cauchy-Schwarz leads to a diagonal term that creates the condition $Q \leq N^{1/2} X^{1/2}$. However since $N$ is no longer tiny this diagonal term is not troublesome. 

 We make two closing remarks. First of all, in \cite[Conjecture 1]{BCR} is stated an optimal conjectural refinement of the bounds of Bettin-Chandee.
 Using this conjectural bound in Section \ref{se:tran} instead of the bounds of Bettin-Chandee leads to an optimal form of Corollary \ref{cor:Main} valid for all $Q \leq X^{1 - \delta}$ and $N < X^{1/2 - \delta}$, for any fixed $\delta > 0$.  The conjecture \cite[Conjecture 1]{BCR} is however very deep; one of its consequences is the Lindel\"of hypothesis. 
%
 %
 Secondly, in the companion paper \cite{Fo-Ra} we investigate another way of applying the Linnik dispersion method. By using the Cauchy--Schwarz inequality in a different way to bound $\vert \Delta \vert$ (see  \eqref{Cauchy} below)  we prove that we can also pass through the barrier  $Q = X^{1/2}$ as soon as $N$ is slightly larger than $X^{1/2}$.
 \\
 \newline
 \noindent{\bf Acknowledgement.} 
 The second author would like to thank the Laboratoire de Math\'{e}matiques
 d'Orsay for its invitation and for its hospitality. The second author also acknowledges support of an NSERC DG grant, the CRC program and a Sloan Fellowship. We would like to thank Sandro Bettin an James Maynard for their comments on the paper. 

\section{Conventions and lemmas}

\subsection{Conventions} Throughout $X$ will stand for $M N$ and $\mathcal{L}$ for $\log (2 M N)$. 
 
The letters $\kappa_1, \kappa_2, \kappa_3, ...$  will denote  functions only depending on the parameter $k$ appearing in the size condition $|\alpha_{m}| \leq \tau_k(m)$ and $|\beta_{n}| \leq \tau_k(n)$. Although  possible, there is no advantage in explicitly writing the values of $\kappa_1$, $\kappa_2$, $\kappa_3$,...
Similarly $C_1$, $C_2$, $C_3$,... will denote absolute positive constants whose value could be specified explicitly but there is (at the moment) no good reason to do so. 

If $f$ is a smooth real function, its Fourier transform is defined by
$$
\hat f (\xi) = \int_{-\infty}^\infty f(t) e( -\xi t) \, {\rm d} t,
$$
where $e(\cdot)= \exp (2 \pi i \cdot).$

\subsection{Lemmas}
Our first lemma is  a  classical finite version of  the Poisson summation formula 
in arithmetic progressions, with a good error term.

\begin{lemma} \label{existenceofpsi}There exists  
  a  smooth function $\psi\ : \ \mathbb R \longrightarrow \mathbb R^+$, compactly supported in $[1/2, 5/2]$ such that $\psi(t) = 1$ for $1 \leq t \leq 2$, and whose derivatives satisfy the inequality
\begin{equation*}
\vert \psi^{(j+1)} (t)\vert \leq  8 (4^j\, j!)^2,
\end{equation*}
for every integer $j$ and  every real $t$.  Furthermore,  uniformly for  integers $a$ and $q \geq 1$,  for  $M \geq 1$ and $H\geq  (q/M)\log ^4 2M$
one has the equality
\begin{equation}\label{ineqpsi}
\sum_{m\equiv a \bmod q} \psi \Bigl( \frac{m}{M}\Bigr) =\hat{\psi}(0)  \frac {M}{q}
+ \frac{M}{q} \sum_{0 < \vert h \vert \leq H}e \bigl( \frac{ ah}{q} \bigr)\hat \psi \Bigl( \frac{h}{q/M}\Bigr) +O(M^{-1}).
\end{equation}
Finally, uniformly for $q\geq 1$ and $M\geq 1$ one has the equality
\begin{equation}\label{Poissoncoprime}
\sum_{(m,q)=1} \psi \Bigl( \frac{m}{M}\Bigr) =\frac{\varphi (q)}{q}\hat{\psi}(0)  M + O \bigl(\tau_2 (q) \log^4 2M \bigr).
\end{equation}

\end{lemma}
\begin{proof} In \cite[Corollary, p. 368]{B-F-I2} such a function $\psi$ (named $\sigma$ there) is built  with $A=1$ and $B=2$ but with the interval $[1/2, 5/2]$ replaced by $[-1, 4]$.  By a re-scaling, we obtain the inequality concerning the derivative  $\psi^{(j+1)}$.
   The equality 
\eqref{ineqpsi} is \cite[Lemma 7]{B-F-I2} with a different normalization. 

Finally \eqref{Poissoncoprime} is a consequence of \eqref{ineqpsi} combined with $$\mathbf{1}_{(m,q) = 1} = \sum_{\substack{d | m \\ d | q}} \mu(d)$$
and the trivial inequality $\vert \hat \psi \vert =O(1)$.
\end{proof}
We cite below a classical upper bound of Shiu's \cite[Theorem 2]{Sh} for multiplicative functions in arithmetic progressions.

\begin{lemma}\label{le:shiu}
  Let $\varepsilon, k > 0$ be given. Let $g : \mathbb{N} \rightarrow \mathbb{C}$ be a multiplicative function such that $|g(n)| \leq \tau_k(n)$. Then, uniformly in $2 \leq x^{\varepsilon} \leq y \leq x$, $q \leq y x^{-\varepsilon}, (a,q) = 1$,
  $$
  \sum_{\substack{x - y \leq n \leq x \\ n \equiv a \pmod{q}}} |g(n)| \ll_{k, \varepsilon} \frac{y}{\varphi(q)} \cdot \prod_{p \leq x} \Big ( 1 + \frac{|g(p)| - 1}{p} \Big ). 
  $$
  \end{lemma}

  An immediate consequence of Shiu's theorem is the following Lemma for the divisor function. We cite this special case since it will be frequently used. For a more precise statement see \cite[Theorem 2]{Sh}
  
\begin{lemma}\label{dkinarith} 
 For every $k\geq 1$,  $\ell\geq 1$, and $\varepsilon >0$, there exists a constant $C(k,\ell, \varepsilon) > 0$
such that for all $x >y> x^\varepsilon$, $1 \leq q \leq y x^{-\varepsilon}$, and every integer $a$ co-prime to $q$ we have, 
$$
\sum_{\substack{x-y <n\leq x \\ n \equiv a \pmod q}} \tau_k^\ell (n) \leq C({k, \ell, \varepsilon})\,\frac{ y}{\varphi(q)} (\log x)^{k^\ell -1}.
$$
\end{lemma}

Our main tool is a bound for trilinear forms in Kloosterman fractions. It is due to Bettin and Chandee  \cite[Theorem 1]{Be-Ch}. Their result has at its origin the paper of Duke, Friedlander and Iwaniec (\cite [Theorem 2]{D-F-I}) which deals with bilinear forms. These two papers show cancellations in  exponential sums involving Kloosterman fractions $a \overline m/n$ with $m \asymp n$. The result of Bettin-Chandee produces extra cancellations when summing over $a$, a feature that is used in the proof of Theorem \ref{Generalresult}.

\begin{lemma}\label{trilinear} Let $\varepsilon > 0$ be given. There exists a constant $C = C(\varepsilon) > 0$ such that for every non-zero integer $\vartheta$, and for every sequence of complex numbers $\boldsymbol \alpha = (\alpha_{m})$, $\boldsymbol \beta = (\beta_{n})$, and $\boldsymbol \nu = (\nu_{a})$, and for every $A, M, N \geq 1$, we have, 
 \begin{multline*}
\Bigl\vert \, \sum_{a\sim A} \sum_{m\sim M} \sum_{n \sim N} \alpha (m) \beta (n) \nu (a) e \Bigl(\vartheta \frac{a \overline m}{n} \Bigr)\,   \Bigr\vert \leq C(\varepsilon) \Vert \boldsymbol \alpha\Vert_{2, M} \, \Vert \boldsymbol \beta\Vert_{2, N} \, \Vert \boldsymbol \nu\Vert_{2, A} 
\\
\times  \Bigl( 1+ \frac{\vert \vartheta\vert A}{MN}\Bigr)^\frac{1}{2}  \Bigl( (AMN)^{\frac{7}{20} + \varepsilon} \, (M+N) ^\frac{1}{4} +(AMN)^{\frac{3}{8} +\varepsilon} (AN+AM) ^\frac{1}{8}
\Bigr). 
 \end{multline*}
 \end{lemma}


\section{Preparation of the dispersion}\label{dispersion}  We now follow the computation of a dispersion as it appears in \cite{FoActaMath}, \cite{BFI1}  and \cite{FoAnnENS} for instance.
If $(q,a) = 1$ define $c_q$ to be a complex number of modulus $1$ such that   $c_{q} E(\boldsymbol \alpha, \boldsymbol \beta, M, N , q, a) = |E (\boldsymbol \alpha, \boldsymbol \beta, M, N, q, a)|$. If $(q,a) > 1$ then set $c_{q} = 0$. Inverting summations in the definition \eqref{defDelta}, applying  the  Cauchy--Schwarz inequality, inserting the $\psi$--function of Lemma \ref{existenceofpsi} and finally expanding the square we see that
\begin{align}\label{Cauchy}
|\Delta (\boldsymbol \alpha, \boldsymbol \beta, M, N, Q, a)| &=
\Big | \sum_{m\sim M}  \alpha_m  \Bigl(  \underset{\substack{q \sim Q\ n\sim N \\ mn \equiv a \bmod q}}{\sum\  \sum } c_q \beta_n  - 
\underset{\substack{q \sim Q\ n\sim N \\ (mn, q)=1}}{\sum\  \sum } \frac{c_q}{\varphi (q)} \beta_n \Bigr)\Big | \nonumber \\
& \leq \Vert \boldsymbol \alpha\Vert_{2,M} \Bigl\{ \sum_{m} \psi \Bigl ( \frac{m}{M}\Bigr ) \Bigl |  \underset{\substack{q \sim Q\ n\sim N \\ mn \equiv a \bmod q}}{\sum\  \sum } c_q \beta_n  -
\underset{\substack{q \sim Q\ n\sim N \\ (mn, q)=1}}{\sum\  \sum } \frac{c_q }{\varphi (q)}\beta_n \Bigr |^2\Bigr\}^\frac{1}{2}\nonumber \\
& \leq  \Vert \boldsymbol \alpha\Vert_{2,M}\Bigl\{ W(Q)- 2 \Re V(Q) +U (Q)\Bigr\}^{\frac 12}, 
\end{align}
where

\begin{enumerate}
  \item the sum  $U(Q)$ is defined by
\begin{equation}\label{defU}
U(Q)= \sum_m\ \psi 
\Bigl( \frac{m}{M} \Bigr)\Bigl | \, \underset{\substack{q \sim Q\ n\sim N \\ (mn, q)=1}}{\sum\  \sum } \frac{c_q }{\varphi (q)}\beta_n \,\Bigr |^2,
\end{equation}
\item the sum $V(Q)$ is defined by 
\begin{equation}\label{defV}
V(Q)= \sum_{m} \psi \Bigl( \frac{m}{M} \Bigr) \Bigl(  \underset{\substack{q \sim Q\ n\sim N \\ mn \equiv a \bmod q}}{\sum\  \sum } c_q \beta_n \Bigr) \overline{\Bigl(  
\underset{\substack{q \sim Q\ n\sim N \\ (mn, q)=1}}{\sum\  \sum } \frac{c_q}{\varphi (q)} \beta_n \Bigr)},
\end{equation}
\item and the most important sum  $W(Q)$ is defined by 
\begin{equation}\label{defW}
W(Q) =  \sum_{m} \psi \Bigl( \frac{m}{M} \Bigr)  \Bigl |  \underset{\substack{q \sim Q\ n\sim N \\ mn \equiv a \bmod q}}{\sum\  \sum } c_q \beta_n \Bigr |^2.
\end{equation}
\end{enumerate}

\section{Study of $U(Q)$} Expanding the square in \eqref{defU} we have the equality
\begin{equation*}
U (Q) =\sum_{q_1\sim Q} \sum_{q_2\sim Q} \frac{c_{q_1}}{\varphi (q_1)} \cdot \frac{\overline{c_{q_2}}}{\varphi (q_2)}\sum_{\substack{n_1\sim N \\ (n_1, q_1)=1}}
\sum_{\substack{n_2\sim N \\ (n_2, q_2)=1}} \beta_{n_1} \overline{\beta_{n_2}} \sum_{(m, q_1q_2)=1} \psi \Bigl( \frac{m}{M}\Bigr),
\end{equation*}
which combined with  \eqref{Poissoncoprime} of Lemma \ref{existenceofpsi} gives  the equality
\begin{multline*}
U (Q) = \hat \psi (0) M  \sum_{q_1\sim Q} \sum_{q_2\sim Q} \frac{c_{q_1}}{\varphi (q_1)} \cdot \frac{\overline{c_{q_2}}}{\varphi (q_2)} \cdot \frac{\varphi (q_1q_2)}{q_1 q_2}\,\sum_{\substack{n_1\sim N \\ (n_1, q_1)=1}}
\sum_{\substack{n_2\sim N \\ (n_2, q_2)=1}} \beta_{n_1} \overline{\beta_{n_2}} \\
+  O\bigl(  \mathcal L^{\kappa_1} N^2 \bigr).
\end{multline*}
It remains to sum over $\delta = (q_1,q_2)$ to get  the final equality
\begin{multline}\label{finalU}
U(Q) = \hat \psi (0) M\sum_\delta \frac{1}{\delta \varphi (\delta)} \underset{\substack{k_1,  k_2 \sim Q/\delta \\ (k_1,k_2)=1}}{\sum\ \sum} \frac{ c_{\delta k_1} \overline{c_{\delta k_2}}}{k_1k_2} \Bigl( \sum_{\substack{n_1 \sim N \\ (n_1, \delta k_1) =1}} \beta_{n_1} \Bigr) \, \overline{\Bigl( \sum_{\substack{n_2 \sim N \\ (n_2, \delta k_2) =1}} \beta_{n_2} \Bigr)}\\
+  O\bigl(   \mathcal L^{\kappa_1} N^2 \bigr).
\end{multline}

\section{Study of $V(Q)$} Expanding the products in \eqref{defV} and inverting summations we obtain the equality 
\begin{equation}\label{V1}
V(Q)=  \sum_{q_1\sim Q} \sum_{q_2\sim Q} c_{q_1} \frac{\overline{c_{q_2}}}{\varphi (q_2)}\sum_{\substack{n_1\sim N \\ (n_1, q_1)=1}}
\sum_{\substack{n_2\sim N \\ (n_2, q_2)=1}} \beta_{n_1} \overline{\beta_{n_2}} \sum_{\substack{m\equiv a \overline{n_1}\bmod q_1 \\ (m,q_2) =1}} \psi \Bigl( \frac{m}{M}\Bigr).
\end{equation}
Let $d$ be an integer such that $d\mid q_2$ and $(d,q_1) =1$. Let $\lambda_0 \bmod dq_1$ be the unique solution of the congruences $\lambda_0 \equiv a\overline{n_1} \bmod q_1$  and $\lambda_0 \equiv 0 \bmod d$.  By \eqref{ineqpsi} of  Lemma \ref{existenceofpsi} we have, for $H=(dq_1\mathcal L^4)/M 
$
 the equality  
\begin{align*}
\sum_{\substack{m\equiv a \overline{n_1}\bmod q_1 \\ d \vert m}} \psi \Bigl( \frac{m}{M}\Bigr) 
&= \hat \psi (0) \frac{M}{dq_1} + \frac{M}{dq_1}\sum_{1\leq \vert h \vert \leq H} e \Bigl( h\frac {\lambda_0}{dq_1}\Bigr) \hat \psi \Bigr( \frac{h}{(dq_1)/M} \Bigr) +O (M^{-1})\\
& = \hat \psi (0) \frac{M}{dq_1} + O ( \mathcal L^4).
\end{align*}
as a consequence of $\vert \hat \psi \vert = O(1).$ By the M\" obius inversion formula, we deduce the equality
$$
\sum_{\substack{m\equiv a \overline{n_1}\bmod q_1 \\ (m,q_2) =1}} \psi \Bigl( \frac{m}{M}\Bigr) =  \hat \psi (0)  \frac{M}{q_1} \sum_{\substack{d\vert q_2 \\ (d, q_1)=1}} \frac{\mu (d)}{d}
+ O \bigl( \tau_2 (q_2) \mathcal L^4 \bigr).
$$
Inserting this equality into \eqref{V1} and introducing $\delta = (q_1, q_2)$, we obtain the equality
\begin{multline*}
V(Q)= \hat \psi (0) M \sum_\delta \frac{1}{\delta \varphi (\delta)} \underset{\substack{k_1 \sim Q/\delta\ k_2 \sim Q/\delta \\ (k_1, k_2)=1}}{\sum\  \ \sum}\frac{c_{\delta k_1} \overline{c_{\delta k_2}}}{k_1 k_2} 
\Bigl( \sum_{\substack{n_1 \sim N\\ (n_1, \delta k_1) =1}} \beta_{n_1} \Bigr) \, \overline{\Bigl( \sum_{\substack{n_2 \sim N\\ (n_2, \delta k_2) =1}}
\beta_{n_2}\Bigr)}  \\ 
+ O \bigl(  \mathcal L ^{\kappa_2} N^2 Q\bigr).
\end {multline*}
Comparing with \eqref{finalU} we obtain the equality 
\begin{equation}\label{V=U}
V(Q) =U(Q) + O \bigl(  \mathcal L ^{\kappa_3}N^2 Q \bigr).
\end{equation}

\section{Study of $W(Q)$} We now turn our attention to the most delicate sum, for which we will appeal to bounds for exponential sums.
Expanding the square in \eqref{defW}
we have
\begin{equation}\label{defW1}  
W(Q)=  \sum_{q_1 \sim Q} \sum_{q_2 \sim Q} c_{q_1} \overline{c_{q_2}}\sum_{n_1 \sim N} \sum_{n_2 \sim N} \beta_{n_1} \overline{\beta_{n_2}} \sum_{\substack{m, mn_1\equiv a \bmod q_1\\ mn_2 \equiv a \bmod q_2}} \psi \Bigl( \frac{m}{M} \Bigr).
\end{equation}
It is worth noticing that since $\psi (t)=0$ out of the interval $[1/2, 5/2]$, since $n\sim N$ and since $\vert a \vert \leq X/3$ we always
have 
\begin{equation}\label{mnn=a}
mn - a\ne 0.
\end{equation}
  This remark will simplify the proof of Lemma \ref{approxW}.

\subsection{Controlling the multiplicative decomposition of the variables.}
In this subsection we want to arithmetically prepare the variables $q_1$, $q_2$, $n_1$ and $n_2$  appearing in \eqref{defW1}, so as to facilitate the application of Lemma \ref{trilinear}.
This preparation is now classical (see for instance \cite[p.235--237]{FoActaMath}).  We adopt the following notational conventions to decompose the variables $q_1, q_2, n_1, n_2$ in a unique way:
\begin{equation}\label{conv}
\begin{cases} d = (n_1,n_2)\\
n_1= d\nu_1, \ n_2 = d \nu_2,\\
\nu_{1}= d_{1}\nu'_{1}\text{ with } d_1 \vert d^\infty \text{  and } (\nu'_1 ,d) =1,\\
\delta = (q_1, q_2),\\
q_1 = \delta k_1, \, q_2 = \delta k_2,\\
k_1= \delta_1 k'_1 \text{ with } \delta_1 \vert \delta^\infty \text{  and } (k'_1 ,\delta) =1,\\
k_2= \delta_2 k'_2 \text{ with } \delta_2 \vert \delta^\infty \text{  and } (k'_2 ,\delta) =1.
\end{cases}
\end{equation}

Since $(a, q_1 q_2) = 1$ and we are summing over integers $m$ such that $mn_1 \equiv a \bmod q_1$ and $mn_2 \equiv a \bmod q_2$ we also have
\begin{equation}\label{coprim}
(dd_1 \nu'_1,\delta \delta_1 k'_1) = (d\nu_2,  \delta \delta_2 k'_2) =1.
\end{equation}
This will be used several times below without further notice.

Given $D, D_{1}, \Delta, \Delta_1, \Delta_2 \geq 1$ let $W(Q,D,D_{1},\Delta, \Delta_1, \Delta_2)$ be the contribution to the right--hand side of \eqref{defW1} of the integers $q_1, q_2, n_1, n_2, m$ satisfying
\begin{equation}\label{extracond}
d\leq D,d_1  \leq D_{1}, \delta\leq \Delta, \, \delta_1 \leq \Delta_1, \, \delta_2 \leq \Delta_2.
\end{equation}
The parameters $D, D_1, \Delta, \Delta_1, \Delta_2$ will be chosen small. In fact to make things simple we will choose them to be equal, and thus set, 
\begin{equation}\label{W(QD)}
W(Q, D) = W(Q, D, D, D, D, D).
\end{equation}
The purpose of the following lemma is to prove that we can approximate $W (Q)$ by $W(Q,  D,D_{1, } \Delta, \Delta_1, \Delta_2)$.
\begin{lemma}\label{approxW}
  Let $k > 0$. Let $\boldsymbol \beta = (\beta_{n})$ be a sequence such that $|\beta_{n}| \leq \tau_k(n)$ for all $n \geq 1$. Then there exists $\kappa = \kappa (k)$ such that,   for every $\varepsilon >0$, for every $D, D_{1},\Delta, \Delta_1, \Delta_2 \geq 1$, one has 
\begin{multline*}
W(Q)- W(Q,D,D_{1 },\Delta, \Delta_1, \Delta_2) \\
=O_{\varepsilon }\Bigl(\mathcal L^\kappa MN^2 (D^{-1} + D_{1}^{-\frac{1}{2}}+ \Delta^{-1} + \Delta_1^{-\frac{1}{2}} + \Delta_2^{-\frac{1}{2}})  \Bigr),
\end{multline*}
uniformly for $M$, $N$ and $Q\geq 2$ satisfying  
$$
M \geq QX^\varepsilon,
\ M > N >\Delta^{10} \text{ and } 
 1 \leq \vert a \vert \leq X/3.
 $$
 In particular we have 
 \begin{equation}\label{W-W(QD)}
 W(Q)-W(Q,D)= O_{\varepsilon }\bigl(\mathcal L^\kappa D^{-\frac{1}{2}} MN^2\bigr)
 \end{equation}
 uniformly for $D$, $M$, $N$ and $Q\geq 2$ satisfying  
$$
M \geq QX^\varepsilon,
\ M >  N >D^{10} \text{ and } 
 1 \leq \vert a \vert \leq X/3.
 $$
\end{lemma}
\begin{proof} The proof is a consequence  of bounds for the divisor function $\tau_k$ in arithmetic progressions  (see Lemma \ref{dkinarith}). We bound
$\vert W(Q)- W(Q,D, D_{1},\Delta, \Delta_1, \Delta_2)\vert$ by noticing that each of the $q_1, q_2, n_1, n_2, m$ that contributes to $|W(Q) - W(Q, D, D_1, \Delta, \Delta_1, \Delta_2)|$ falls into one of the five cases below. We then estimate the contribution of each case. 
\begin{enumerate}
\item {\it the contribution of $q_1, q_2, n_1, n_2, m$ with $d> D$.} The contribution of such $q_1, q_2, n_1, n_2, m$  is less than (recall \eqref{mnn=a})
\begin{align*} 
 & \sum_{d >D}  \sum_{\nu_1 \sim N/d} \sum_{\nu_2\sim N/d} \tau_k (d\nu_1) \tau_k (d \nu_2) \sum_{M/2 <m <5 M/2  } \tau_2 (\vert dm\nu_1-a\vert )\tau_2 (\vert dm\nu_2-a\vert )\\
& \ll  \mathcal L^{\kappa_4}  M N^2 \sum_{d > D}  \tau_k (d)^2d^{-2}
 \ll \mathcal L^{\kappa_5}\,D^{-1}MN^{2} ,
 \end{align*}
\item {\it the contribution of $q_1, q_2, n_1, n_2, m$ with $d \leq D$ and $d_{1}> D_{1}$}. The contribution of such $q_1, q_2, n_1, n_2, m$ is bounded by
\begin{align*} 
 & \sum_{d \leq D} \sum_{\substack{d_{1}>D_{1} \\ d_1 | d^{\infty}}} \sum_{\nu'_1 \sim N/(dd_1)} \sum_{\nu_2\sim N/d} \tau_k (dd_{1}\nu'_1) \tau_k (d \nu_2)
  \\ & \qquad \qquad \qquad  \times \sum_{M/2 <m <5 M/2 } \tau_2 (\vert mdd_{1}\nu'_1-a\vert )\tau_2 (\vert md\nu_2-a\vert )
 \\
& \ll  \mathcal L^{\kappa_6}  M N^2 \sum_{d \leq D}  \tau_k (d)^2 d^{-2} \sum_{d_{1 }\vert d^\infty} \Bigl( \frac{d_{1}}{D_{1}}\Bigr)^\frac{1}{2}  \frac{\tau_k(d_1)}{d_{1}} \ll \mathcal L^{\kappa_7}\,D_{1}^{-\frac{1}{2}}MN^{2} ,
 \end{align*}
\item {\it the contribution of $q_1, q_2, n_1, n_2, m$ with $\delta > \Delta$.} The conditions of summation over $m$ in \eqref{defW1} imply that we necessarily have
\begin{equation}\label{n1congn2}  
n_1 \equiv n_2 \bmod \delta,
\end{equation} 
hence, thanks to the assumption $M> QX^\varepsilon$, and using the trivial bound $\tau_k(n_2) \ll N^{1/10}$, the contribution of integers $q_1, q_2, n_1, n_2, m$ with $\delta > \Delta$ is in absolute value less than 
\begin{align*}
& \sum_{\delta > \Delta}  \sum_{k_2 \sim Q/\delta}\  
\underset{\substack{n_1 \sim N,\ n_2 \sim N \\  n_1 \equiv n_2 \bmod \delta}}{\sum \ \sum}  \tau_k (n_1) \tau_k (n_2)
\sum_{\substack{M/2 <m <5 M/2  \\ m\equiv a \overline{n_2} \bmod \delta  k_2}} \tau_2 (\vert mn_1-a\vert )  \\
& \ll  \mathcal L^{\kappa_8} \Big ( \sum_{\Delta < \delta < N^\frac{1}{4}}    \sum_{k_2 \sim Q/\delta}  \sum_{n_1 \sim N} \tau_k (n_1)   \frac{N}{\delta}   \cdot  \frac{M}{\delta k_2} \\ & \qquad \qquad \qquad + N^\frac{1}{10}\sum_{\substack{N M \geq \delta >N^{1/4}}}   \sum_{k_2 \sim Q/\delta}  \sum_{n_1 \sim N} \tau_k (n_1)  \Bigl( \frac{N}{\delta}  +1 \Bigr)   \cdot  \frac{M}{\delta k_2} 
\Bigr ) \\
& \ll \mathcal L^{\kappa_9} \bigl( MN^2 \Delta^{-1} + MN^\frac{37}{20} \bigr) \ll  \mathcal L^{\kappa_9} MN^2 \Delta^{-1},
\end{align*}
\item {\it the contribution of $q_1, q_2, n_1, n_2, m$ with $\delta \leq \Delta$ and $\delta_1 > \Delta_1$.}  Thanks to the assumptions
$M>QX^\varepsilon$ and  
$N > \Delta^{10}$, the contribution of these integers is less than
\begin{align*}
&\sum_{\delta \leq \Delta} \sum_{\substack{\delta_1 \vert \delta^{\infty} \\ \delta_1 > \Delta_1}} \ \sum_{k'_1 \sim Q/(\delta \delta_1)}\  \underset{\substack{n_1 \sim N,\ n_2 \sim N \\  n_1 \equiv n_2 \bmod \delta}}{\sum \ \sum}  \tau_k (n_1) \tau_k (n_2)
\sum_{\substack{m\sim M \\ m\equiv a \overline{n_1} \bmod \delta  \delta_1 k'_1}} \tau_2 (\vert mn_2-a\vert ) \\ & \ll \mathcal L^{\kappa_{10}}  \sum_{\delta \leq \Delta} \sum_{\substack{\delta_1 \vert \delta^{\infty} \\ \delta_1 > \Delta_1}} \ \sum_{k'_1 \sim Q/(\delta \delta_1)}  N \cdot \frac{N}{\delta} \cdot \frac{M} {\delta \delta_1 k'_1} \ll \mathcal L^{\kappa_{11}}  MN^2 \sum_{\delta \leq \Delta} \sum_{\substack{\delta_1 \vert \delta^{\infty} \\ \delta_1 > \Delta_1}}\frac{1}{\delta^2 \delta_1}  \\
&\ll \mathcal L^{\kappa_{12}} MN^2 \sum_{\delta \leq \Delta} \frac{1}{\delta^2}  \sum_{\delta_1\vert \delta^\infty} \frac{1}{\delta_1} \Big ( \frac{\delta_1}{\Delta_1} \Big )^\frac{1}{2} \ll \mathcal L^{\kappa_{13}}  MN^2 \Delta_1^{-\frac{1}{2}}.
\end{align*}
\item {\it   the  contribution of $q_1, q_2, n_1, n_2, m$ with $\delta \leq \Delta$ and $\delta_2 > \Delta_2$.} Proceeding in the same way as above we find that the contribution of such integers is
$$
\ll \mathcal L^{\kappa_{14}}  MN^2 \Delta_2^{-\frac{1}{2}}.
$$
\end{enumerate}
Gathering the five upper bounds proved above, we complete the proof of Lemma \ref{approxW}.
 \end{proof}

 \subsection{Preparing the  congruences.} We keep  the notations introduced in \eqref{conv} and the co-primality conditions from \eqref{coprim}. Our immediate goal is to transform the system of  congruences  appearing in \eqref{defW1}, namely,
 \begin{equation}\label{system}
 \begin{cases}
 m \equiv a \overline{n_1} \bmod q_1\\
 m\equiv a \overline{n_2} \bmod q_2
 \end{cases}
 \end{equation}
 into a single congruence.   We assume that \eqref{n1congn2} is satisfied
 hence the system \eqref{system} is solvable.  By \eqref{coprim}  we deduce that \eqref{system} is equivalent to the system of four congruences
 \begin{equation}\label{system2}
 \begin{cases}
 m\equiv a \overline {n_1} \bmod \delta \delta_1, &  m\equiv a \overline {n_2} \bmod \delta \delta_2,\\
 m\equiv  a \overline {n_1} \bmod k'_1, &  m\equiv a \overline {n_2} \bmod k'_2.
 \end{cases}
 \end{equation}
 The first two equations are equivalent to $m\equiv a \lambda \bmod \delta \delta_1 \delta_2$, where $\lambda (n_1, n_2)$ is some congruence class modulo $\delta \delta_1 \delta_2$, only depending on the congruence classes of $n_1 \bmod \delta \delta_1$ and $n_2 \bmod \delta \delta_2$. Finally, we see that \eqref{system2} is equivalent to the unique equation 
 \begin{equation}\label{m=m0}
 m\equiv m_0 \bmod \ell,
 \end{equation}
 with 
 \begin{equation}\label{defell}
 \ell  = \gamma k'_1 k'_2,
 \end{equation}
 \begin{equation}\label{defgamma}
 \gamma = \delta \delta_1 \delta_2,
 \end{equation}
 and  $$
 m_0 = a \lambda \,k'_1k'_2\, \overline{ k'_1} \, \overline{k'_2} 
 +a\, \gamma\,\overline{\gamma} \,\overline{n_1}\,k'_2 \, \overline{k'_2}  +a\, \gamma\, \overline{\gamma}\, \overline{n_2} \, k'_1 \overline{k'_1}
 \bmod \ell,
 $$
 where the $\overline{x}$--symbol  respectively means the inverse of $x$  modulo $\gamma$, $k'_{1}$ and $k'_{2}$.
 In the sequel of the proof, we will have in mind that $\delta, \delta_1, \delta_2, \gamma, d, d_1$ are tiny variables, $\nu'_1, \nu_2$ are small variables and, finally, $k'_1$ and $k'_2$ are large variables which will produce cancellations when summing over them.

 \subsection{Expanding in Fourier series.}  The last sum over $m$ in \eqref{defW1} is precisely 
 $$
 \sum_{m\equiv m_0 \bmod \ell} \psi \Bigl( \frac{m}{M}\Bigr).
 $$
 Fix 
 \begin{equation}\label{defH}
H= 4 \mathcal L^4 Q^2 M^{-1}  \ \ (>\ell \mathcal L^4 M^{-1} ).
\end{equation}
By \eqref{ineqpsi} of Lemma \ref{existenceofpsi} the above  sum over $m$  is equal to
\begin{equation}\label{Fou1}
\hat{\psi}(0)  \frac {M}{\ell}
+ \frac{M}{\ell} \sum_{1 \leq \vert h \vert \leq H}e \bigl( \frac{ hm_0}{\ell} \bigr)\hat \psi \Bigl( \frac{hM}{\ell}\Bigr)+O (M^{-1}).
\end{equation}
 Inserting this equality in  the definition of $W(Q,D)  $ (see \eqref{defW1},  \eqref{extracond} and \eqref{W(QD)}) we split $W(Q,D)$ into three parts  corresponding to each of the components of  the sum \eqref{Fou1}
\begin{equation}\label{splitW}
W(Q,D)= W^{\rm MT} + W^{\rm Err1}+ W^{\rm Err2}, 
\end{equation}
where $W^{\rm MT} $ is the main term, $ W^{\rm Err1} $ is the delicate error term corresponding to the sum over $1 \leq |h| \leq H$ and $W^{\rm Err2}$ is the trivial error term corresponding to the total contribution of the error term $O(M^{-1})$ appearing in \eqref{Fou1}. In particular, we have
\begin{equation}\label{Err2}
W^{\rm Err2} \ll   \mathcal L ^{\kappa_{15}} M^{-1} N^2 Q^2.
\end{equation}
\subsection{Dealing with the main term $W^{\rm MT}$ }
By definition we have the equality
\begin{multline*}
W^{\rm MT}
 =
\hat \psi (0)  M\underset{d, d_1, \delta, \delta_1, \delta_2 \leq D}{ \sum \cdots \sum}  \ \ \   \sum_{ k'_1 \sim Q/(\delta \delta_1)\ }
\sum_{ k'_2 \sim Q /( \delta \delta_2)}
\frac{c_{\delta \delta_1k'_1} \overline{c_{\delta \delta_2 k'_2}}}{\delta \delta_1 \delta_2 k'_1k'_2}\\
\sum_{\nu'_1 \sim Q/(dd_1)} \sum_{\nu_2 \sim Q/d}\beta_{dd_1\nu'_1} \beta_{d \nu_2},
\end{multline*}
where the variables of summation  satisfy  the divisibility  and co-primality conditions appearing in \eqref{conv} and \eqref{coprim} and the congruence condition  (see \eqref{n1congn2})
\begin{equation} \label{eq:congonce}
dd_1\nu'_1 \equiv d\nu_2 \bmod \delta.
\end{equation}
We now drop the conditions $d, d_1, \delta, \delta_1, \delta_2 \leq D$, this generates an error  of the shape $O( \mathcal L^{\kappa_{16}} 
D^{-\frac 12}MN^2)$ (the computations are similar to those made in the proof of Lemma \ref{approxW}).
We now sum over all the reduced congruences $\alpha \bmod \delta$. This gives the equality
\begin{multline*}
W^{\rm MT}
 =
\hat \psi (0)  M\underset{d, d_1, \delta, \delta_1, \delta_2 }{ \sum \cdots \sum}  \ \ \   \sum_{ k'_1 \sim Q/(\delta \delta_1)\ }
\sum_{ k'_2 \sim Q /( \delta \delta_2)}
\frac{c_{\delta \delta_1k'_1} \overline{c_{\delta \delta_2 k'_2}}}{\delta \delta_1 \delta_2 k'_1k'_2}\\
 \sum_{\substack{\alpha \mod \delta\\ (\alpha, \delta)=1}}
\Bigl(\sum_{\substack{\nu'_1 \sim N/(dd_1)\\ dd_1\nu'_1\equiv \alpha \bmod \delta }}  \beta_{dd_1 \nu_1} \Bigr) 
\overline{\Bigl(\sum_{\substack{\nu_2 \sim N/d_2\\ d\nu_2\equiv \alpha \bmod \delta }}  \beta_{d\nu_2} \Bigr)} +
O( \mathcal L^{\kappa_{16}} 
D^{-\frac 12}MN^2).
\end{multline*}
  Returning to the original variables  $k_1$, $k_2$, $n_{1}$ and $n_{2}$
we obtain  the equality
\begin{multline}\label{MTW}
W^{\rm MT}
 =
\hat \psi (0)  M  \sum_{\delta }   \underset{\substack{k_{1} \sim Q/\delta \ , \ k_2 \sim Q/\delta \\ (k_1, k_2) = 1}}{\sum \sum}
\frac{c_{\delta k_{1}} \overline{c_{\delta k_{2}}}}{\delta  k_1k_2}\\ \sum_{\substack{\alpha \mod \delta\\ (\alpha, \delta)=1}}
\Bigl(\sum_{\substack{n_{1}\sim N, (n_{1}, \delta k_1) =1\\ n_1\equiv \alpha \bmod \delta}}  \beta_{n_1} \Bigr) 
\overline{\Bigl(\sum_{\substack{n_{2} \sim N, (n_{2}, \delta k_2 )=1\\ n_{2}\equiv \alpha \bmod \delta}}  \beta_{n_{2}} \Bigr)}+O( \mathcal L^{\kappa_{16}} 
D^{-\frac 12}MN^2).
\end{multline}
\subsection{Transformation of the exponential sum.} \label{se:tran} We now turn our attention to the error term $W^{\rm Err1}$.  By \eqref{defW1}, \eqref{Fou1} and \eqref{splitW}, the error term $W^{\rm Err1}$ is defined by 
\begin{equation}\label{defWErr1}
W^{\rm Err1} =  M \sum_{q_1 \sim Q} \sum_{q_2 \sim Q} \frac{c_{q_1} \overline{c_{q_2}}}{\ell} \sum_{n_1 \sim N} \sum_{n_2 \sim N} \beta_{n_1} \overline{\beta_{n_2}}   \sum_{1 \leq \vert h \vert \leq H}e \bigl( \frac{ hm_0}{\ell} \bigr)\hat \psi \Bigl( \frac{hM}{\ell}\Bigr),
\end{equation}
where the variables above are still subject to the conventions set out in \eqref{conv}, \eqref{coprim},  \eqref{defell},  \eqref{defgamma},  \eqref{defH}, and the inequalities
\begin{equation}\label{<D}
d,d_1, \delta, \delta_1, \delta_2 \leq D.
\end{equation}
and the congruence condition \eqref{eq:congonce}.
We now exploit that, since $h$ is small, the Fourier transform $\hat \psi (hM/\ell)$ fluctuates slowly in the following sense:
There exists an absolute $C_1 > 0$ such that
$$
\frac{ \partial^{a_1+a_2+a_3}}{\partial h^{a_1}\partial {k'}_1^{a_2} \partial {k'}_2^{a_3}}\Bigl\{\,\frac{Q}{k'_1}\, \frac{Q}{k'_2}\hat \psi \Bigl( \frac{hM}{\gamma k'_1 k'_2}\Bigr)\, \Bigr\}
\ll (1 + \vert h\vert)^{-a_1}{k'}_1^{-a_2} {k'}_2^{-a_3} \cdot D^{C_1}
$$
for integers $0 \leq a_1, a_2, a_3 \leq 1$ and real numbers $k_1', k_2' \in [Q / D^{100}, 2Q]$ and real numbers $h \in [-H, H]$.
This bound allows to  integrate by parts to suppress  the coefficient $\hat \psi ( h M / \ell )$ in \eqref{defWErr1}. This shows that,
\begin{multline*}
W^{\text{Err1}} \ll  D^{C_2} \mathcal{L}^{100} M Q^{-2} \\  \times \sup_{\substack{(d, d_1, \delta, \delta_1, \delta_2)}}   
\Bigl\vert\sum_{1 \leq \vert h \vert \leq H}   \sum_{\nu'_1\leq 2N} \sum_{\nu_2\leq 2N}
 \sum_{k'_1\leq 2Q} \sum_{k'_2\leq 2Q} \xi_1 (h) \xi_2 (\nu'_1) \xi_3 (\nu_2) \xi_4 (k'_1) \xi_5 (k'_2)
 e(\cdots)\Bigr\vert, 
 \end{multline*}
 for some $C_2 > 0$, and
 where the supremum is taken over all $(d, d_1, \delta, \delta_1, \delta_2)$ obeying the co-primality conditions implied by \eqref{conv},\eqref{coprim} and the size conditions \eqref{<D}. Moreover $\xi_1, \xi_2, \xi_3, \xi_4, \xi_5$ are sequences of complex numbers of modulus one. Finally, 
$$
e(\cdots)=
e\Bigl(a\lambda h \frac{\overline{k'_1}\, \overline{ k'_2}}{\gamma}
+ah \frac{ \overline \gamma \,\overline {d}\, \overline{d_1}\, \overline {\nu'_1}\,  \overline{k'_2}}{k'_1} 
+  ah \frac{ \overline \gamma\, \overline {d }\, \overline{\nu_2} \, \overline{k'_1}}{k'_2} 
\Bigr). 
$$
We note that fixing the congruence class of $\nu'_1, \nu_2, k'_1, k'_2$ modulo $\gamma d d_1$ obviously also fixes it modulo $\gamma$, and as a result the value of the fraction
$$a \lambda h \frac{\overline{k'_1}\, \overline{ k'_2}}{\gamma}$$ is constant modulo $1$.

It follows that we can further bound $W^{\text{Err1}}$ by
\begin{multline} \label{hammer}
W^{\text{Err1}}  
\ll D^{C_3} \mathcal{L}^{100} M Q^{-2} \\
\times \sup_{\substack{(\alpha_1, \alpha_2, \alpha_3, \alpha_4) \\ (d, d_1, \delta, \delta_1, \delta_2)}} \Big | \underset{\substack{1 \leq |h| \leq H \ , \ \nu'_1 \leq 2 N \\ \nu_2 \leq 2N \ , \ k'_1 \leq 2Q \\ k'_2 \leq 2 Q}}{\sum \ldots \sum} \xi_{1}(h) \xi_{2}(\nu'_1) \xi_{3}(\nu_2) \xi_4(k'_1) \xi_5(k'_2) e_{1}(\ldots) \Big |
\end{multline}
for some $C_3 > 0$ and where the supremum is now taken over all $(\alpha_1, \alpha_2, \alpha_3, \alpha_4)$ belonging to the interval $[0, \gamma d d_1]$ and all $(d, d_1, \delta, \delta_1, \delta_2)$ satisfying the congruence conditions implied by \eqref{conv}, \eqref{coprim} and the size conditions \eqref{<D}. In the summation we impose the additional condition that each $\nu'_{1}, \nu_{2}, k'_1, k'_2$ is in a congruence class $\alpha_1 \pmod{\gamma d d_1}, \ldots, \alpha_{4} \pmod{\gamma d d_1}$, respectively, and as usual the variables $\nu'_1, \nu_2, k'_1, k'_2$ obey the co-primality conditions \eqref{conv} and \eqref{coprim}. Finally, $e_1(\cdots)$ is defined by
\begin{equation}\label{defe1}
e_1 (\cdots) = e\Bigl( 
ah \frac{ \overline \gamma \,\overline {d}\, \overline{d_1}\, \overline {\nu'_1}\,  \overline{k'_2}}{k'_1} 
+  ah \frac{ \overline \gamma\, \overline {d }\, \overline{\nu_2} \, \overline{k'_1}}{k'_2} 
\Bigr).
\end{equation}
 In order  to transform $e_1(\cdots)$  we apply Bezout's relation twice to write
\begin{align*}
 \frac{  \overline \gamma \,\overline {d}\, \overline{d_1}\, \overline{\nu'_1}\,  \overline{k'_2}}{k'_1} 
  &= \frac{1}{\gamma dd_1 \nu'_1 k'_1 k'_2}
  - \frac{\overline{k'_1}}{ \gamma
 dd_1 \nu'_1  k'_2}
 \bmod 1\\
 & = \frac{1}{\gamma dd_1 \nu'_1 k'_1 k'_2 }
 - \frac{\overline{ \gamma} \,  \overline{d}\, \overline{d_1} \,   \overline{ k'_1}}{\nu'_1 k'_2} 
 - \frac{\overline{\nu'_1} \,\overline{k'_1} \, \overline{k'_2}}{ \gamma d d_1 } \bmod 1.
 \end{align*}
 Inserting this expression into \eqref{defe1}  then in \eqref{hammer} we see that 
 \begin{multline}\label{hammer1}
   W^{\rm Err1} \ll D^{C_3} \mathcal{L}^{100} MQ^{-2} \sup_{\substack{(\alpha_1, \ldots, \alpha_4) \\ (d, d_1, \delta, \delta_1, \delta_2)}} \, \Bigl\vert
 \underset{\substack{1 \leq |h| \leq H \\ \nu'_1, \nu_2 \leq 2N \\ k'_1, k'_2 \leq 2Q}}{\sum \ldots \sum}
 \xi_1 (h) \xi_2 (\nu'_1) \xi_3 (\nu_2) \xi_4 (k'_1) \xi_5 (k'_2)\\ 
 \times e\Bigl( ah \bigl(   \frac{1}{\gamma dd_1 \nu'_1 k'_1 k'_2 }
-   \frac{\overline{\nu'_1} \,\overline{k'_1} \, \overline{k'_2}}{\gamma d d_1 } 
+ \frac{ \overline{ \gamma} \, \overline d \, \overline {d_1}\,  (d_1 \nu'_1 -\nu_2) \overline{(\nu_2k'_1)} }{\nu'_1 k'_2}  \bigr)\, 
\Bigr)  \Bigr\vert. 
 \end{multline}
 The first term inside $e(\ )$ is controlled by summation by parts  over the five variables $h$, $\nu'_1$, $\nu_2$, $k'_1$, $k'_2$ since
 $1\leq \vert a \vert \leq X/3$. The second term in $e(\ )$ is fixed $\bmod\, 1$ since  the congruence classes of  $h$, $\nu'_1$, $k'_1$ and $k'_2$ are fixed modulo $\gamma d d_1$. 
From \eqref{hammer1} we get the inequality
\begin{multline}\label{hammer3}
 W^{\rm Err1} \ll D^{C_4} \mathcal{L}^{100} MQ^{-2} \sup_{\substack{(\alpha_1, \ldots, \alpha_5) \\ (d, d_1, \delta, \delta_1, \delta_2)}} \sum_{\nu'_1\leq 2N} \sum_{\nu_2\leq 2N}    \Bigl\vert  \sum_{1 \leq h \leq H} 
 \sum_{k'_1\leq 2Q} \sum_{k'_2\leq 2Q}  \eta_0 (h)  \eta_1(k'_1) \eta_2 (k'_2)\\ 
 e\Bigl( ah 
 \frac{ \overline{ \gamma} \, \overline d \, \overline {d_1}\,  (d_1 \nu'_1 -\nu_2) \overline{(\nu_2k'_1)} }{\nu'_1 k'_2}  \bigr)\, 
\Bigr) \Bigr\vert,
 \end{multline}
 where $\eta_0 (h)$,  $\eta_1 (k'_1)$ and $\eta_2 (k'_2)$ are  some unspecified coefficients  less than or equal to $1$ in modulus. Notice that this allows us to without loss of generality replace the condition $1 \leq |h| \leq H$ by $1 \leq h \leq H$. 
 To bound \eqref{hammer3}, we will sum trivially over $\nu'_1$ and $\nu_2$ and use Lemma \ref{trilinear} on the remaining variables.

 We localize each of the variables $\nu'_1, \nu_2, h, k'_1, k'_2$ dyadically around powers of two that we denote respectively by $N_1, N_2, H_{d}, K_1, K_2$.
 On each such dyadic partition we apply Lemma \ref{trilinear} with the following choice of variables, 
 $$
 \vartheta \rightarrow a(d_1\nu'_1 -\nu_2), \ a\rightarrow \underline{h}, \ m\rightarrow \gamma dd_1\nu_2 \underline{k'_1} ,\ n\rightarrow \nu'_1 \underline{k_2} ,
 $$
 (where we underlined the non-fixed variables and where the left-side of $\rightarrow$ corresponds to notations of Lemma \ref{trilinear} while the right-side of $\rightarrow$ corresponds to our current notation)
 and parameters
 $$
 \vert \vartheta \vert \ll \vert a \vert D N_1 + |a| N_2, \ A\rightarrow  H, \  M \rightarrow \gamma d d_1 \nu_2 K_1 , \ N \rightarrow \nu_1' K_2. 
 $$
 Note that $\gamma d d_1 \nu_2 K_1 \leq D^{5} N_2 K_1$ and that $\nu_1' K_2 \leq N_1 K_2$. 
 The values of the corresponding  $\ell_2$--norms
 are respectively 
 $$
  \Vert \boldsymbol \alpha\Vert_{2, M} \ll K_1^{1/2} \ , \ \Vert \boldsymbol \beta\Vert_{2, N} \ll K_2^{1/2} \text{ and }\Vert \boldsymbol \nu\Vert_{2, A}\ll H_{d}^\frac{1}{2}.
  $$
  Returning to \eqref{hammer3}   we deduce that
  \begin{align*}
 &  W^{\rm Err1} \ll \mathcal{L}^{200} \sup_{(N_1, N_2, K_1, K_2, H_{d})} D^{C_5}  M Q^{-2} N_1 N_2 (K_1 K_2 H_{d})^\frac{1}{2} \Bigl( 1 + \frac{ \vert a \vert N_1 \cdot H_{d} + |a| N_2}{N_1 N_2 K_1 K_2} \Bigr)^\frac{1}{2} \cdot\\ &
\Big( ( H_{d} N_1 N_2 K_1 K_2)^{\frac{7}{20} +\varepsilon} (N_2 K_1 + N_1 K_2)^\frac{1}{4} + ( H_{d} N_1 N_2 K_1 K_2)^{\frac{3}{8} +\varepsilon} (H_{d} (N_2 K_1 + N_1 K_2)  )^\frac{1}{8}
   \Big)
  \end{align*}
  for  some constant $C_5 > 0$ and where the supremum runs over all powers of two $N_1, N_2, K_1, K_2, H_{d}$ obeying the conditions $1 \leq N_1, N_2 \leq 2N$, $1 \leq K_1, K_2 \leq 2Q$ and $1 \leq H_{d} \leq H$. 
 Summing over all the dyadic partitions that were available to us we get, 
 \begin{align*}
 W^{\rm Err1} \ll \mathcal{L}^{200} & D^{C_5} M Q^{-2} N^{2} \cdot Q H^{1/2} \cdot \Big ( 1 + \frac{|a| }{M N} \Big )^{\frac{1}{2}} \\ & \times \Big ( (H N^2 Q^2)^{\frac{7}{20} + \varepsilon} \cdot ( N Q)^{\frac{1}{4}} + (H N^2 Q^2)^{\frac{3}{8} + \varepsilon} \cdot (H N Q)^{\frac 18} \Big ).
 \end{align*}
Recalling \eqref{defH},  and the inequality $1\leq \vert a \vert \leq X/3$  this bound simplifies 
  into 
   \begin{align}
   W^{\rm Err1} &\ll D^{C_5} X^\varepsilon M^\frac{1}{2} N^2 \bigl( M^{-\frac{7}{20}} N^\frac{19}{20} Q^\frac{33}{20} +  M^{-\frac{1}{2}}
   N^\frac{7}{8} Q^\frac{15}{8}\bigr)\nonumber \\
   &\ll  D^{C_5} X^\varepsilon \bigl( M^\frac{3}{20} N^\frac{59}{20}  Q^\frac{33}{20}+ N^\frac{23}{8} Q^\frac{15}{8}\bigr).\label{WErr1<<}
   \end{align}
 

\subsection{The main term}  Gathering \eqref{finalU},  \eqref{V=U} and   \eqref{MTW}  we see that the main term of $W (Q)-2 \Re V(Q)+ U(Q)$ is equal to 
\begin{equation}\label{defT(Q)}
  T (Q):=M \hat \psi (0) \sum_{\delta} \frac{1}{\delta}\underset{(k_1,k_2) =1} {\sum \ \sum} \frac{c_{\delta k_1} \overline{c_{\delta k_2}}}{k_1\, k_2} 
\sum_{\substack{\delta'\bmod \delta\\ (\delta', \delta) =1}}
E^{\star}(\boldsymbol \beta, N, \delta, \delta' ;k_1)  \overline{E^{\star}( \boldsymbol \beta, N, \delta,  \delta'; k_2)}
\end{equation}
where the function $E^{\star}$ is defined in \eqref{defEmodif}.  Since $\vert c_q\vert\leq 1$, we deduce that
$$
\vert T(Q) \vert \ll M   \sum_{\delta} \frac{1}{\delta}\underset{ k_1, \ k_2 \sim Q/\delta } {\sum \ \sum} \frac{1}{k_1\, k_2} 
\sum_{\substack{\delta'\bmod \delta\\ (\delta', \delta) =1}}
\vert E^{\star}(\boldsymbol \beta, N, \delta, \delta'; k_1)\vert ^2,
$$
which finally gives
$$
\vert T(Q) \vert \ll M   \sum_{\delta} \frac{1}{\delta}\underset{ k \sim Q/\delta } { \sum} \frac{1}{k} 
\sum_{\substack{\delta'\bmod \delta\\ (\delta', \delta) =1}}
\vert E^{\star}(\boldsymbol \beta, N, \delta, \delta'; k)\vert ^2,
$$
which gives
\begin{align}
\vert T(Q) \vert &\ll MQ^{-1}   \sum_{\delta} \     \sum_{k\sim Q/\delta } \sum_{(\delta', \delta) =1}
\vert E^{\star}(\boldsymbol \beta, N, \delta, \delta'; k)\vert ^2\nonumber\\
& \ll MQ^{-1}  \mathcal E^{\star}(\boldsymbol \beta, N, Q),\label{end!!}
\end{align}
by the definition \eqref{defcalE}.

 Gathering   \eqref{Cauchy}, \eqref{V=U},  \eqref{W-W(QD)}, \eqref{splitW},  \eqref{Err2},  \eqref{MTW}, \eqref{WErr1<<} and \eqref{end!!} completes the proof of  Theorem \ref{Generalresult}.
  
 \section{Proof of Theorem \ref{NOS-W}}\label{NoSiegel}  In the statement of Theorem \ref{NOS-W} we only sum over primes $q\sim Q$.
 We follow through the proof of Theorem \ref{Generalresult} with a different definition of $c_{q}$: If $q$ is prime and belongs to $[Q,2Q]$ then set $c_{q}$ to be a complex number of modulus one such that $c_{q} E(\boldsymbol \alpha, \boldsymbol \beta, M, N, q, a) = |E(\boldsymbol \alpha, \boldsymbol \beta, M, N, q, a)|$ and if $q$ falls outside of the range $[Q, 2Q]$ or if $q$ is not prime then set $c_{q} = 0$. The same proof goes through up until the point where we reach the expression $T(Q)$ defined in \eqref{defT(Q)}. The analysis of $T(Q)$ proceeds now as follows.

 We notice that the presence of $c_{\delta k_1} \overline{c_{\delta k_2}}$ means that either $\delta = 1$ and $k_1 \neq k_2$ are prime, or $\delta$ is prime and $k_1 = k_2 = 1$. In the first case we get,
 $$
 M \hat{\psi}(0) \sum_{\substack{k_1, k_2 \sim Q \\ k_1 \neq k_2 \\ \text{ prime }}} \frac{c_{k_1} \overline{c_{k_2}}}{k_1 k_2} E^{\star}(\boldsymbol \beta, N, 1, 1; k_1) \overline{E^{\star}(\boldsymbol \beta, N, 1, 1; k_2)}  = 0
 $$
since we trivially have $E^{\star}(\boldsymbol \beta, N, 1, 1; k) = 0$ for any $k \geq 1$. 
 In the second case we get
 \begin{equation} \label{eq:secondcasebound}
 M \hat{\psi}(0) \sum_{\substack{\delta \sim Q \\ \text{ prime }}} \frac{1}{\delta} \sum_{\substack{(\delta', \delta) = 1}} |E^{\star}(\boldsymbol \beta, N, \delta', \delta; 1)|^2.
 \end{equation}
 By the orthogonality of characters
 $$
 \sum_{(\delta',\delta) = 1} |E^{\star}(\boldsymbol \beta, N, \delta', \delta; 1)|^2 = \frac{1}{\varphi(\delta)} \sum_{\chi \neq \chi_{0} \pmod{\delta}} \Big | \sum_{n \sim N} \beta_{n} \chi(n) \Big |^2 .
 $$
 Thus we find that \eqref{eq:secondcasebound} equals
 \begin{align*}
 M \hat{\psi}(0) \sum_{\substack{\delta \sim Q \\ \delta \text{ is prime}}} & \frac{1}{\delta (\delta - 1)} \sum_{\chi \neq \chi_{0} \pmod{\delta}} \Big | \sum_{n \sim N} \beta_n \chi(n) \Big |^2 \ll \frac{M}{Q^2} \cdot (N + Q^2) \| \boldsymbol \beta \|^{2}_{2, N} \\ & \ll M N (\log N)^{k^2 - 1} + M N^2 Q^{-2} (\log N)^{k^2 - 1}.  
 \end{align*}
 as a consequence of the large sieve inequality, of the inequality $|\beta_n| \leq \tau_k(n)$ and of Lemma \ref{dkinarith}. Since $Q \geq \exp(\sqrt{\log N})$ we conclude that,
 $$
 T(Q) \ll M N^2 \exp(- \sqrt{\log N}) (\log N)^{k^2 - 1} 
 $$
Moreover since $Q \leq N^{-\frac{12}{11}} X^{\frac{17}{33} -\varepsilon}$, we have
$$ M^\frac{3}{20} N^\frac{59}{20} Q^\frac{33}{20} \leq MN^2 X^{-\varepsilon} 
\text{ and  } N^\frac{23}{8} Q^\frac{15}{8} \leq  N^\frac{73}{88} X^{\frac{85}{88}-\varepsilon} < MN^2 X^{-\varepsilon}.
$$
Therefore Theorem \ref{NOS-W} follows.

 \section{Proof of Corollary \ref{cor:Main}}\label{Proofofcorollary1.1}

 We will prove Corollary \ref{cor:Main} by breaking down the proof into two stages.

 We first require the following Lemma, showing that if $\boldsymbol \beta$ is Siegel-Walfisz then $\mathcal{E}^{\star}(\boldsymbol \beta, N, Q)$ is small. 

 \begin{lemma}\label{874}
   Let $k\geq 1$ be an integer. Let $\boldsymbol \beta = (\beta_{n})_{n \sim N}$ be a sequence of complex numbers such that $|\beta_{n}| \leq \tau_k(n)$ for all $n \geq 1$. Suppose that $\boldsymbol \beta$ is Siegel-Walfisz. Then for every $N\geq 1$, for every $Q\geq 1$ and for every $B$, we have
 $$ \mathcal E^{\star} (\boldsymbol \beta, N, Q) \ll_{B} N^2 \, Q\,  (\log 2N)^{-B } (\log 2Q)^{k^2-1} + N^\frac{7}{4} Q.
 $$
 \end{lemma}
 \begin{proof} This proof already appears in \cite[p. 242--243]{FoActaMath}. Recall that (see \eqref{defcalE})
$$   \mathcal{E}^{\star}( \boldsymbol \beta, N, Q) := \sum_{\delta} \sum_{v \sim Q / \delta} \sum_{(\delta', \delta) = 1} |E^{\star} ( \boldsymbol \beta, N, \delta, \delta'; v)|^2. 
$$
We use different types of bounds for the term $E^{\star}(\beta, N, \delta, \delta '; v)$ defined in  \eqref{defEmodif} according to the size of $\delta$. Let $A_1$ be a parameter whose value we will fix later. Let $0 < \varepsilon < \tfrac{1}{1000}$ be given.

\begin{enumerate}
  \item If $1 \leq \delta \leq (\log N)^{A_1}$, then since $\boldsymbol \beta$ is Siegel-Walfisz, 
 \begin{equation}\label{case1}
 E^{\star}(\boldsymbol \beta, N, \delta, \delta '; v) \ll N (\log 2 N)^{-A} \tau_k (v),
 \end{equation}
 with $A$ arbitrary but fixed.
\item If $(\log N)^{A_1} < \delta \leq N^\frac{1}{2}$, then by the bound $|\beta_{n}| \leq \tau_k(n)$ and Lemma \ref{dkinarith}, we have
 \begin{equation}\label{case2}
 E^{\star}(\boldsymbol \beta, N, \delta, \delta '; v) \ll \delta ^{-1} N (\log 2 N)^{k-1}.
\end{equation}
\item If $N^\frac{1}{2} < \delta \leq 2N$, then since $|\beta_{n}| \leq \tau_k(n) \ll n^{\varepsilon}$, we trivially have 
 \begin{equation}\label{case3}
 E^{\star} (\boldsymbol \beta, N, \delta, \delta '; v) \ll \delta^{-1} N^{1 + \varepsilon}.
 \end{equation}
\item If $ 2N < \delta \leq 2Q$, we use the trivial bound
 \begin{equation}\label{case4}
 E^{\star} (\boldsymbol \beta, N, \delta, \delta '; v) \ll N^\varepsilon,
\end{equation}
and conclude that, in this range the subsum $\mathcal{E}^{\star}$ with $2 N \leq \delta \leq 2Q$
is bounded by
$$
\ll N^{\varepsilon} \sum_{\delta} \sum_{v \sim Q / \delta} \sum_{n \sim N} \tau_k(n) \ll N^{1 + 2 \varepsilon} Q 
$$
\end{enumerate}
Combining the above bounds we obtain
 \begin{align*}
 \mathcal E^{\star} (\boldsymbol \beta, N, Q)  & \ll  N^2 Q (\log 2N)^{A_1-2A}\ (\log 2Q)^{k^2-1}
 + N^2  Q\  (\log N)^{2k-2 -A_1} +  N^\frac{7}{4} Q + N^{1 + 2\varepsilon} Q.
 \end{align*}
 which gives the result with the choice 
 $
 A= \frac{A_1 + B}{2} \text{ and }  A_1= B+2k-2  .
 $
 \end{proof}

 Before proving Corollary \ref{cor:Main} let us establish the following slightly weaker statement.

 \begin{proposition} \label{pr:prep}
  Let $k > 0$ be an integer and let $\varepsilon > 0$ be given. Let $\boldsymbol \alpha = (\alpha_{m})_{m \sim M}$ and $\boldsymbol \beta = (\beta_{n})_{n \sim N}$ be two sequences of complex numbers such that $|\alpha_{m}| \leq \tau_k(m)$ and $|\beta_{n}| \leq \tau_k(n)$ for all $m,n \geq 1$. Let $X = M N$. Suppose that $\boldsymbol \beta$ is Siegel-Walfisz. Then for every $A > 0$ we have,
  \begin{equation} \label{eq:result}
  \sum_{\substack{Q \leq q \leq 2Q \\ (q,a) = 1}} \Big | \sum_{\substack{m n \equiv a \pmod{q}}} \alpha_{m} \beta_n - \frac{1}{\varphi(q)} \sum_{\substack{(m n, q) = 1}} \alpha_m \beta_{n} \Big | \ll_{A} X (\log X)^{-A}
  \end{equation}
  provided that any of the following three conditions holds
  \begin{enumerate}
  \item $\exp((\log X)^{\varepsilon}) \leq N \leq Q^{- 11 / 12} \cdot X^{17/36 - \varepsilon}$ and $1 \leq |a| \leq X / 3 $
  \item $\exp((\log X)^{\varepsilon}) \leq N \leq X^{7/90 - \varepsilon}$, $Q \leq X^{53/105 - \varepsilon}$ and $1 \leq |a| \leq X / 3 $
  \item $\exp((\log X)^{\varepsilon}) \leq N \leq X^{101/630 - \varepsilon}$, and $Q \leq X^{53/105 - \varepsilon}$ and $1 \leq |a| \leq X^{\varepsilon / 1000}$. 
    \end{enumerate}
 
 \end{proposition}

 \begin{proof}

   We start by establishing part $i)$. We fix 
 $$
 D = \exp (\mathcal L^\varepsilon/10) \text{ with } \mathcal{L} = \log 2X,$$ 
and we apply Lemma  \ref{dkinarith}  to obtain the  bound $\Vert \boldsymbol \alpha\Vert_{2,M} \ll M^{1/2} \mathcal L^{(k^2-1)/2}$.
Furthermore  when $Q \leq N^{-\frac{12}{11}} X^{\frac{17}{33} -\varepsilon}$, we have
$$ M^\frac{3}{20} N^\frac{59}{20} Q^\frac{33}{20} \leq MN^2 X^{-\varepsilon} 
\text{ and  } N^\frac{23}{8} Q^\frac{15}{8} \leq  N^\frac{73}{88} X^{\frac{85}{88}-\varepsilon} < MN^2 X^{-\varepsilon}.
$$
Moreover since $\boldsymbol \beta$ is Siegel-Walfisz by Lemma \ref{874} we find that
$$
\mathcal{E}^{\star}(\boldsymbol \beta, N, Q) \ll N^2 Q (\log N)^{-A} .
$$
Combining these bounds and Theorem \ref{Generalresult} gives $i)$.

In order to show $ii)$ we appeal to a result of Fouvry  \cite[Corollaire 1]{FoAnnENS}, according to which we have \eqref{eq:result} for $1 \leq |a| \leq X / 3$ provided that,
$$
Q \leq \min(\sqrt{N} \sqrt{X}, N^{-6/7} X^{4/7}) X^{-\varepsilon} \text{ and } N > X^{\varepsilon}.  
$$
Suppose thus that $Q \leq X^{53/105 - \varepsilon}$. If $N \leq X^{1/105 - \varepsilon}$ then \eqref{eq:result} follows from $i)$. For the remaining values $X^{1/105 - \varepsilon} < N < X^{7/90 - \varepsilon}$ we appeal to Fouvry's result  \cite[Corollaire 1]{FoAnnENS}.

In order to show $iii)$ we appeal to another result of Fouvry  \cite[Th\'eor\`eme 1]{FoActaMath}, according to which we have \eqref{eq:result} for $1 \leq |a| \leq X^{\varepsilon / 1000}$ provided that,
$$
Q \leq \min(\sqrt{N} \sqrt{X}, N^{-3/4} X^{5/8}) X^{-\varepsilon} \text{ and } N > X^{\varepsilon}.
$$
Then, once again for $N < X^{1/105 - \varepsilon}$ we use $i)$. If $X^{1/105 - \varepsilon} < N < X^{101/630}$ then the result follows from Fouvry's result  \cite[Th\'eor\`eme 1]{FoActaMath}. 





\end{proof}

We are now finally ready to prove Corollary \ref{cor:Main}. This amounts to adding the condition $x < m n \leq 2x$ to Proposition \ref{pr:prep}.
\begin{proof}[Proof of Corollary \ref{cor:Main}]
  Let $X = M N$. 
  Without loss of generality we can assume, 
$$ X/2\leq x\leq 4X,
$$
otherwise our sum is zero and the bound trivial. Let
$
\Delta = \mathcal L^{-B}
$ for some $B > 0$ that will be fixed later. 
Let $f$ a  fixed smooth function with support equal to $[1-\Delta, 2 + \Delta]$ with value equal to $1$ on $[1,2]$ and with the derivatives satisfying $\sup_{x \in \mathbb{R}} \vert f^{(k)}(x) \vert \ll_k \Delta^{-k}$ for all integer $k$.  Let 
\begin{equation}\label{defEf}
E_f(\boldsymbol \alpha, \boldsymbol \beta, M, N , x, q, a):= \underset{\substack{m\sim M,\ n \sim N \\ mn \equiv a \bmod q}}{\sum \ \ \ \sum} \alpha_m \beta_n f \Bigl( \frac{mn}{x}\Bigr) -\frac{1}{\varphi (q) }\underset{\substack{m\sim M,\ n \sim N \\ (mn ,q)=1}}{\sum \ \ \ \sum} \alpha_m \beta_n f \Bigl( \frac{mn}{x}\Bigr).
\end{equation}
Since $|\alpha_{m}| \leq \tau_k(m)$ and $|\beta_{n}| \leq \tau_k(n)$ for all $m,n \geq 1$, we have the inequality
\begin{multline*}
\sum_{\substack{q\sim Q\\ (q,a)=1}} \Bigl\vert \, E(\boldsymbol \alpha, \boldsymbol \beta, M, N, x  , q, a) -
E_f(\boldsymbol \alpha, \boldsymbol \beta, M, N, x  , q, a)\, \Bigr\vert \\
\leq\sum_{q\sim Q} \Bigl(  \sum_{\substack{ (1-\Delta) x\leq \ell < x\\ \ell \equiv a \bmod q}}  +   \sum_{\substack{ 2 x\leq \ell < (2+ \Delta) x\\ \ell \equiv a \bmod q}}\Bigr) \tau_{2k} (\ell) \\ +  \sum_{q\sim Q}\frac{1}{\varphi (q)} \Bigl(  \sum_{ (1-\Delta) x\leq \ell < x}  +   \sum_{ 2 x\leq \ell < (2+ \Delta) x }\Bigr) \tau_{2k} (\ell), 
\end{multline*}
and by  Lemma \ref{dk}  we  finally deduce
\begin{equation}\label{1051}
\ll X \sum_{q\sim Q} \frac{1}{\varphi (q)} \mathcal L ^{-B} \, \mathcal L^{2k-1} \ll X \mathcal L^{-A},
\end{equation}
with the choice $B = 2k+A-1.$ 
The Mellin transform $\widetilde f (s)$ of $f(u)$, defined by
$$
\widetilde f (s):= \int_0^\infty f(u) u^{s-1} \ {\rm d} u,
$$
is defined for $s \in \mathbb C$, and it satisfies the decay property
\begin{equation}\label{bounddotf}
\widetilde f (it) \ll\min \bigl(\, 1,  \bigl( \Delta (1+ \vert t \vert )\bigr)^{-2}\,\bigr).
\end{equation}
Inserting the inversion formula
$$
f\Bigl(\frac{mn}{x}\Bigr)= \frac{1}{2\pi } \int _{-\infty}^{\infty} \widetilde f(it) \Bigl( \frac{mn}{x}\Bigr)^{-it}\, {\rm d} t,
$$
in the definition \eqref{defEf},  we have the equality
$$
E_f(\boldsymbol \alpha, \boldsymbol \beta, M, N , x, q, a)= \frac{1}{2\pi}
\int_{-\infty}^\infty \widetilde f (it ) x^{it} E(\boldsymbol \alpha_t, \boldsymbol \beta_t, M, N ,  q, a) \ {\rm d} t,
$$
where  $E(\boldsymbol \alpha_t, \boldsymbol \beta_t, M, N , x, q, a)$ is defined as   $E(\boldsymbol \alpha,  \boldsymbol \beta, M, N , x, q, a)$ but with $\alpha_m$ replaced by $\alpha_m m^{-it}$ and $\beta_n $ by $\beta_n n^{-it}$. Hence by \eqref{1051}, the proof is reduced to establishing the bound
\begin{equation}\label{1073}
\int_{-\infty}^\infty  \vert \widetilde f (it) \vert \sum_{\substack{q\sim Q\\ (q,a)=1}}
 \bigl\vert E(\boldsymbol \alpha_t, \boldsymbol \beta_t, M, N ,  q, a) \, \bigr\vert\ {\rm d} t \ll MN \mathcal L^{-A}.
 \end{equation}
 The trivial bound 
 $$
 \sum_{\substack{q\sim Q \\ (q,a)=1}} \bigl\vert\ E(\boldsymbol \alpha_t, \boldsymbol \beta_t, M, N ,  q, a)\, \bigr\vert \ll MN \mathcal L^{2k}
 $$
 and the bound \eqref{bounddotf} allows to reduce  the integration in \eqref{1073} to the segment
 \begin{equation}\label{condfort}
 \vert t \vert \leq \mathcal L^{A+2B+2k}= \mathcal L^{3A +6k - 2}.
 \end{equation}
 Secondly, integrating by parts we see that uniformly for $t$ as above, the sequence $\boldsymbol \beta_t$ is Siegel-Walfisz. 
 Then by Proposition \ref{pr:prep}
 $$
 \bigl\vert E(\boldsymbol \alpha_t, \boldsymbol \beta_t, M, N ,  q, a) \, \bigr\vert\  \ll_{A'} MN \mathcal L^{-A'},
 $$
 for all $A'$, uniformly for $t$ satisfying \eqref{condfort}. Integrating over $t$ on this interval, we complete the proof of \eqref{1073}.

\end{proof}
 

\section{Proof Corollary \ref{cor:mult}}

Let $\varepsilon > 0$ and $A > 2$ be given. For every $x > 2$ there exists a real number $\Delta$ satisfying the inequalities $(\log x)^{-A} \leq \Delta \leq 2 (\log x)^{-A}$, such that the number
$$
L_0 := \frac{\varepsilon \log x - (\log x)^{\varepsilon}}{\log(1 + \Delta)}
$$
is an integer. Notice that $L_0 \ll (\log x)^{A + 1}$.

Let $\mathcal{S} := \mathcal{S}_{\varepsilon, A, x} \subset [x, 2x]$ be a subset of the integers such that
\begin{enumerate}
\item Each $n \in \mathcal{S}$ has a prime factor $p$ with $p \in \mathcal{J} := [\exp((\log x)^{\varepsilon}), x^{\varepsilon}]$
\item Each $n \in \mathcal{S}$ has at most one prime factor in each of the intervals
  $$
  \mathcal{I}_{\ell} := [\exp((\log x)^{\varepsilon}) (1 + \Delta)^{\ell}, \exp((\log x)^{\varepsilon}) (1 + \Delta)^{\ell + 1}) ) := [M_{\ell}, M_{\ell + 1})
  $$
  with $0 \leq \ell < L_0$. 
\end{enumerate}

Notice that for each $n \in \mathcal{S}_{\varepsilon, A, x}$ there exists a unique $0 \leq \ell < L_0$ such that we can write $n = b p c$ such that all of the prime factors of $b$ are strictly less than $\exp((\log x)^{\varepsilon})$, $p$ belongs to $\mathcal{I}_{\ell}$ and all of the prime factors of $c$ are greater than $M_{\ell + 1}$. 


The contribution to the left-hand side of \eqref{eq:cormult} of the integers that do not belong to $\mathcal{S}$ is negligible as shown in the lemma below. 

\begin{lemma} \label{le:trivialS}
  Let $\varepsilon, k > 0$ and $A > 2$ be given. Let $g : \mathbb{N} \rightarrow \mathbb{C}$ be a multiplicative function such that $|g(n)| \leq \tau_k(n)$. Then, uniformly in $x \geq 2$, $q \leq x^{3/4}, (a,q) = 1$, 
  $$
  \Big | \sum_{\substack{n \sim x \\ n \equiv a \pmod{q} \\ n \not \in \mathcal{S}_{\varepsilon, A, x}}} g(n) \Big | \ll_{\varepsilon, k} \frac{x}{\varphi(q)} \cdot \Big ( (\log x)^{-1 + \varepsilon k} + (\log x)^{k - A} \Big ).
  $$
\end{lemma}
\begin{proof}
  By the union bound, 
\begin{align} \label{eq:trivialtobound} 
\Big | \sum_{\substack{n \sim x \\ n \equiv a \pmod{q} \\ n \not \in \mathcal{S}_{\varepsilon, A, x}}} g(n) \Big | & \leq \sum_{\substack{n \sim x \\ n \equiv a \pmod{q} \\ p | n \implies p \not \in \mathcal{J}}} |g(n)| & + \sum_{0 \leq \ell \leq L_0} \sum_{\substack{n \sim x \\ n \equiv a \pmod{q} \\ r, s \in \mathcal{I}_{\ell} \\ r s | n \ , \ r ,s  \text{ prime}}} |g(n)| 
\end{align}
By Shiu's bound (Lemma \ref{le:shiu}) the first sum is 
$$
\ll \frac{x}{\varphi(q)} \cdot \frac{1}{\log x} \cdot \prod_{p \leq \exp((\log x)^{\varepsilon})} \Big ( 1 + \frac{|g(p)|}{p} \Big ) \prod_{x^{\varepsilon} \leq p \leq x} \Big ( 1 + \frac{|g(p)|}{p} \Big )
\ll_{\varepsilon} \frac{x}{\varphi(q)} \cdot (\log x)^{\varepsilon k - 1}
$$
To bound the second sum on the right-hand side of \eqref{eq:trivialtobound} we notice that for $n = r s m$ we have $|g(n)| \leq \tau_k(n) \leq \tau_{k}(r) \tau_{k}(s) \tau_k(m) \ll_{k} \tau_k(m)$. Notice also that $(n,q) = 1$ since $n \equiv a \pmod{q}$ and $(a,q) = 1$. Therefore we also have $(rs, q) = 1$. Consequently, 
\begin{align*}
\sum_{0 \leq \ell \leq L_0} \sum_{\substack{ n \sim x \\ n \equiv a \pmod{q} \\ r, s \in \mathcal{I}_{\ell} \\ r s | n \ , \ r ,s  \text{ prime}}} & |g(n)| \ll_{k} \sum_{0 \leq \ell \leq L_0} \sum_{\substack{m \leq x / (rs) \\ r,s \in \mathcal{I}_{\ell} \\ (rs, q) = 1 \\ m \equiv  a \overline{r s} \pmod{q}}} \tau_k(m) \\ &  \ll_{k} \sum_{0 \leq \ell < L_0} \Big ( \sum_{\substack{r, s \in \mathcal{I}_{\ell} \\ r,s \text{ prime}}} \frac{1}{r s} \Big ) \cdot \frac{x}{\varphi(q)} \cdot (\log x)^{k - 1} \\ & \ll_{k} L_0  \frac{x}{\varphi(q)} \cdot (\log x)^{k - 2A - 1} \ll \frac{x}{\varphi(q)} \cdot (\log x)^{k - A}
\end{align*}

Combining the above two bounds the claim follows. 
\end{proof}

As a consequence of Lemma \ref{le:trivialS} the proof of the Corollary will be finished once we show that for every $D > 0$, 
$$
\sum_{q \sim Q} \Big | \sum_{\substack{n \sim x \\ n \equiv a \pmod{q} \\ n \in \mathcal{S}_{\varepsilon, A, x}}} g(n) - \frac{1}{\varphi(q)} \sum_{\substack{n \sim x \\ (n,q) = 1 \\ n \in \mathcal{S}_{\varepsilon, A, x}}} g(n) \Big | \ll_{D} x (\log x)^{-D} 
$$
Using the definition of the set $\mathcal{S}_{\varepsilon, A, x}$ and the triangle inequality, we bound this as, 
\begin{equation} \label{eq:bilinear}
\sum_{0 \leq \ell < L_0} \sum_{q \sim Q} \Big | \sum_{\substack{p m \sim x \\ p m \equiv a \pmod{q} \\ p \in \mathcal{I}_{\ell}, m \sim x / M_{\ell + 1} }} \alpha_{p,\ell,x} \beta_{m,\ell,x} 
- \frac{1}{\varphi(q)} \sum_{\substack{p m \sim x \\ (pm, q) = 1 \\ p \in \mathcal{I}_{\ell} ,m \sim x / M_{\ell + 1}}} \alpha_{p,\ell,x} \beta_{m,\ell,x} \Big | 
\end{equation}
where 
$$
\alpha_{p,\ell,x} := g(p) \cdot \mathbf{1}_{p \in \mathcal{I}_{\ell}}
$$
and
$$
\beta_{m,\ell,x} := \sum_{\substack{m = b c \\ p | b \implies p < M_{0} \\ p | c \implies p > M_{\ell + 1}}} g(b) g(c) .
$$
Notice that $\alpha_{p,\ell,x}$ is Siegel-Walfisz, since $g(p)$ is Siegel-Walfisz and since $M_{\ell} \geq \exp((\log x)^{\varepsilon})$ and $\Delta \geq (\log x)^{-A}$. By Corollary \ref{cor:Main}$i)$ the expression \eqref{eq:bilinear} is $\ll_{C} x (\log x)^{-C}$ for any fixed $C > 0$, as needed.

\section{Proof of Corollary \ref{dk}}

Corollary \ref{dk} follows immediately from Corollary \ref{cor:mult} upon specializing to $g(n) = \tau_k(n)$. 

\section{Proof of Corollary \ref{corollary4}}

The proof of Corollary \ref{corollary4} is similar to the proof of Corollary \ref{cor:mult}, but requires slightly more precise estimates. Consequently we provide the proof in full. 

Let $\varepsilon > 0$ and $k \in [2, \infty] \cap \mathbb{N}$ be given. As usual denote by $\Omega(n)$ the number of prime factors of $n$ counted with multiplicity. Let $x \geq 2$ be a large number. Let $1 \leq c_0 = c_0(x) \leq 2$ be a real number such that $c_0 x^{\varepsilon} \exp(-(\log x)^{\varepsilon}) = 2^{J_0}$ where $J_0$ is a positive integer. To shorten notations we write,
$$
y := \exp((\log x)^{\varepsilon^2}).
$$

Let $\mathcal{S} := \mathcal{S}_{\varepsilon, k, x} \subset [x, 2x]$ be the set of those $n = p_1 \ldots p_k$ with $p_1 \leq p_2 \leq \ldots \leq p_k$ for which \begin{enumerate}
\item $p_1 \in [y, c_0 x^{\varepsilon}]$. 
\item there is at most one $p_i$ in each interval $[y 2^{j}, y 2^{j + 1})$ with $0 \leq j < J_0$.
\end{enumerate}

We first claim that the integers $x < n \leq 2x$ with $\Omega(n) = k$ that are not in $\mathcal{S}$ give a negligible contribution to the left-hand side of \eqref{eq:corollary4}. 
\begin{lemma}
Let $k \geq 2$ be an integer and let $0 < \varepsilon < \frac{1}{100 k}$ be given. Then, uniformly in $x \geq x_{0}(\varepsilon,k)$, $q \leq x^{3/4}$ and $(a,q) = 1$, we have 
$$
\sum_{\substack{n \sim x \\ n \equiv a \pmod{q} \\ \Omega(n) = k \\ n \not \in \mathcal{S}_{\varepsilon, k , x}}} 1 \ll_{k} \varepsilon \cdot \frac{x}{\varphi(q)} \cdot \frac{(\log\log x)^{k - 1}}{\log x} .
$$
\end{lemma}
\begin{proof}
If $x < n \leq 2x$ satisfies $\Omega(n) = k$ and $n \not \in \mathcal{S}_{\varepsilon, k, x}$ then either all of the prime factors of $n$ are larger than $c_0 x^{\varepsilon}$ \textbf{or} $n$ has a prime factor smaller than $y$ in which case it can be written as $n = p b c$ with $p < y$, $\Omega(b) \leq k - 1$ with $p | b \implies p \leq c_0 x^{\varepsilon}$ and $p | c \implies p > c_0 x^{\varepsilon}$ \textbf{or} there exists an $ 0 \leq \ell < J_0$ for which there exists two $i < j$ with $p_i, p_j \in [y 2^{\ell}, y 2^{\ell + 1})$. 

Therefore by the triangle inequality, 
$$
\sum_{\substack{n \sim x \\ \Omega(n) = k \\ n \not \in \mathcal{S}_{\varepsilon, k, x} \\ n \equiv a \pmod{q}}} 1 \ll \sum_{\substack{n \sim x \\ p | n \implies p > c_0 x^{\varepsilon} \\ n \equiv a \pmod{q}}} 1 + \sum_{\substack{p b c \sim x , p \leq y \\  p b c \equiv a \pmod{q} \\ \Omega(b) \leq k - 1 \\ p | b \implies p \leq c_0 x^{\varepsilon} \\ p | c \implies p > c_0 x^{\varepsilon}}} 1 + \sum_{\substack{1 \leq i < j \leq k \\ 0 \leq \ell < L_0}} \sum_{\substack{p_1 \ldots p_k \sim x \\ p_i, p_j \in [y 2^{\ell}, y 2^{\ell + 1}] \\ p_1 \ldots p_k \equiv a \pmod{q}}} 1
$$
By the upper bound sieve the contribution of the first sum is
$$
\ll \frac{x}{\varphi(q) \varepsilon} \cdot \frac{1}{\log x}
$$
and therefore acceptable. To deal with the second sum we first remark that we necessarily have $c > 1$ (as a consequence of the sizes of the variables $p$ and $b$). This implies the inequality $\Omega(b) \leq k - 2$. Applying the upper bound sieve again we conclude that the contribution of the second sum is
$$
\ll \sum_{\substack{p \leq y,  b \leq (c_0 x^{\varepsilon})^{k-2} \\ \Omega(b) \leq k - 2 \\ (p b, q) = 1}} \sum_{\substack{p | c \implies p \geq c_0 x^{\varepsilon} \\ c \leq x / (b p) \\ c \equiv a \overline{b p} \pmod{q}}} 1 \ll \frac{x}{\varphi(q) \varepsilon} \cdot \frac{1}{\log x} \sum_{\substack{p \leq y \\ \Omega(b) \leq k - 2 \\ b \leq (c_0 x^{\varepsilon})^{k-2} }} \frac{1}{b p}
$$
and this in turn is
$$
\ll \varepsilon \cdot \frac{x}{\varphi(q)} \cdot \frac{(\log\log x)^{k - 1}}{\log x}.
$$
Finally the contribution of the third sum is bounded by 
$$
\sum_{\substack{1 \leq i < j \leq k \\ 0 \leq \ell < J_0}} \sum_{p_i, p_j \in [y 2^{\ell}, y 2^{\ell + 1})} \sum_{\substack{n \sim x / (p_i p_j) \\ \Omega(n) = k - 2 \\ n \equiv a \overline{p_i p_j} \pmod{q}}} 1
$$
Since $p_i p_j \ll x^{2 \varepsilon}$ we can use the Brun-Titchmarsh theorem to bound the inner sum over $n$, and thus that when $k \geq 2$ the above is 
\begin{align*}
\ll \frac{x}{\varphi(q)} & \frac{(\log \log x)^{k - 3}}{\log x} \sum_{\substack{1 \leq i < j \leq k \\ 0 \leq \ell < J_0}} \sum_{p_i, p_j \in [y 2^{\ell}, y 2^{\ell + 1})} \frac{1}{p_i p_j} \\ & \ll \frac{x}{\varphi(q)} \cdot \frac{(\log\log x)^{k - 3}}{\log x} \sum_{\substack{1 \leq i < j \leq k \\ 0 \leq \ell < J_0}} \frac{1}{(y 2^{\ell})^2} \cdot \frac{(y 2^{\ell})^2}{(\log (y 2^{\ell}))^2},
\end{align*}
which finally gives the bound,
$$
\ll \frac{x}{\varphi(q)} \cdot \frac{(\log\log x)^{k - 3}}{(\log x)^{1 + \varepsilon^2}}
$$
for any fixed $k \geq 2$. Combining the bounds for the above three sums completes the proof of the Lemma. 

\end{proof}

Thus to conclude the proof of the corollary it remains to show that, for any $A > 0$, 
\begin{equation} \label{eq:bilineartobound}
\sum_{q \sim Q} \Big | \sum_{\substack{n \sim x \\ n \equiv a \pmod{q} \\ \Omega(n) = k \\ n \in \mathcal{S}_{\varepsilon, k, x} }} 1 - \frac{1}{\varphi(q)} \sum_{\substack{n \sim x \\ (n,q) = 1 \\ \Omega(n) = k \\ n \in \mathcal{S}_{\varepsilon, k, x}}} 1 \Big | \ll_{A} x (\log x)^{-A}. 
\end{equation}
By definition of the set $\mathcal{S}_{\varepsilon, k, x}$ there exists a unique $0 \leq \ell < J_0$ such that we can write $n = p m$ with $p \in [y 2^{\ell}, y 2^{\ell + 1})$ and $p' | m \implies p' \geq y 2^{\ell + 1}$. Therefore we bound the left-hand side of \eqref{eq:bilineartobound} by
$$
\sum_{0 \leq \ell < J_0} \sum_{0 \leq i \leq 1} \sum_{q \sim Q} \Big | \sum_{\substack{p \sim y 2^{\ell}, m \sim x / ( y 2^{\ell + i}) \\ p m \sim x \\ p m \equiv a \pmod{q} \\ p' | m \implies p' \geq y 2^{\ell + 1}}} 1 - \frac{1}{\varphi(q)} \sum_{\substack{p \sim y 2^{\ell},  m \sim x / (y 2^{\ell + i}) \\ p m \sim x \\ (pm , q) = 1 \\ p' | m \implies p' \geq y2^{\ell + 1}}} 1 \Big | 
$$
Since the sequence of prime satisfies the Siegel-Walfisz condition (by the Siegel-Walfisz theorem!) it follows from Corollary \ref{cor:Main}$i)$ that the above expression is $\ll_{A} x (\log x)^{-A}$ for any fixed $A > 0$. 

Combining both estimates we conclude that for each $0 < \varepsilon < \frac{1}{100 k}$, there exists an $x_0(\varepsilon, k)$ such that for all $x > x_{0}(\varepsilon, k)$ we have, 
$$
\sum_{\substack{q \sim Q \\ (q,a) = 1}} \Big | \sum_{\substack{n \sim x \\ n \equiv a \pmod{q} \\ \Omega(n) = k }} 1 - \frac{1}{\varphi(q)} \sum_{\substack{n \sim x \\ (n,q) = 1 \\ \Omega(n) = k }} 1 \Big | \leq C \varepsilon \cdot x \cdot \frac{(\log\log x)^{k - 1}}{(k - 1)! \log x}
$$
with $C > 0$ an absolute constant. This is exactly the definition of being
$$o \Big ( \frac{x (\log\log x)^{k - 1}}{\log x} \Big )$$
and so the corollary follows.

\section{Proof of Corollary \ref{cor:dirichlet}}

The proof of Corollary \ref{cor:dirichlet} relies on divisor switching and the use of our main Corollary \ref{cor:Main} after the divisor switching is accomplished. 

We are interested in bounding,  
$$
\sum_{\substack{q \sim Q \\ (q,a) = 1}} \Big | \sum_{\substack{d b \sim x \\ d \leq z \\ d b \equiv a \pmod{q}}} \lambda_d - \frac{1}{\varphi(q)} \sum_{\substack{d b \sim x \\ d \leq z \\ (db,q) = 1}} \lambda_d \Big |
$$
To proceed further we would like to be able to assume, without loss of generality, that $(d, a) = 1$. This will be useful once we arrive to divisor switching. Notice that we can split the sum according to $\Delta | a$ such that $(a, d) = \Delta$. Notice also that since $(a,q) = 1$ we have $(\Delta, q) = 1$. Consequently we can re-write the above sum as,
$$
\sum_{\Delta | a } \sum_{\substack{q \sim Q \\ (q,a) = 1}} \Big | \sum_{\substack{d b \sim x / \Delta \\ d \leq z / \Delta \\ d b \equiv (a / \Delta) \pmod{q} \\ (d, a / \Delta) = 1}} \lambda_{d \Delta} - \frac{1}{\varphi(q)} \sum_{\substack{d b \sim x / \Delta \\ d \leq z / \Delta \\ (db,q) = 1 \\ (d, a / \Delta) = 1}} \lambda_{d \Delta} \Big |
$$
We split the sum over $\Delta$ into two sub-sums, those corresponding to $\Delta \leq x^{\varepsilon^2}$ and those corresponding to $\Delta > x^{\varepsilon^2}$.

We first focus on the part of the sum with $\Delta > x^{\varepsilon^2}$. We use Lemma \ref{dkinarith} to bound the contribution of terms with $x^{\varepsilon^2} \leq \Delta < x Q^{-1 - \varepsilon}$, and the trivial bound $\lambda_d \ll d^{\varepsilon}$ to bound the contribution of terms with $\Delta > x Q^{-1 - \varepsilon}$. This shows that the subsum over $\Delta > x^{\varepsilon^2}$ contributes at most
$$
\ll \sum_{\substack{\Delta | a , q \sim Q \\ x^{\varepsilon^2} \leq \Delta \leq x Q^{-1 - \varepsilon}}} \frac{x \tau_k(\Delta)}{\Delta \varphi(q)} \cdot (\log x)^{k} + \sum_{\substack{\Delta | a \\ \Delta > x Q^{-1 - \varepsilon}}} Q^{1 + \varepsilon} x^{\varepsilon} \ll x^{1 - \varepsilon^2 / 2}  + Q^{1 + \varepsilon} x^{2\varepsilon}
$$
and therefore is acceptable.

On the other hand since for all $1 \leq |a| \leq x$ we have $\sum_{\Delta | a} 1/\Delta \ll \log x$ to deal with the contribution of terms with $\Delta \leq x^{\varepsilon^2}$ it's enough to show that,
\begin{equation} \label{eq:extra}
\sum_{\substack{q \sim Q \\ (q,a) = 1}} \Big | \sum_{\substack{d b \sim x \\ d \leq z \\ d b \equiv a \pmod{q} \\ (d,a) = 1}} \lambda_d - \frac{1}{\varphi(q)} \sum_{\substack{db  \sim x \\ d \leq z \\ (db ,q ) = 1 \\ (d,a) = 1}} \lambda_d \Big | \ll x (\log x)^{-A} 
\end{equation}
for arbitrary coefficients $\lambda_d$ with $|\lambda_d| \leq \tau_k(d)$ and 
uniformly in either \begin{enumerate}
\item $z \leq x^{53/105 - \varepsilon / 2}, x^{1 - \varepsilon / 2} > Q > x^{529 / 630 + \varepsilon / 2}$ and $1 \leq |a| \leq x^{\varepsilon / 1000}$
\item \textbf{or} $z \leq x^{53 / 105 - \varepsilon / 2}$, $x^{1 - \varepsilon / 2} > Q > x^{83 / 90 + \varepsilon / 2}$ and $1 \leq |a| \leq x^{1 - 3 \varepsilon}$
\item \textbf{or} $z \leq x^{1/2 + \delta - \varepsilon / 2}$, $x^{1 - \varepsilon / 2} > Q > x^{(71 + 66 \delta)/72 + \varepsilon / 2}$ and $1 \leq |a| \leq x^{1 - 2 \varepsilon}$ for any fixed $0 < \delta < \frac{1}{66}$. 
\end{enumerate}


We re-write \eqref{eq:extra} as, 
\begin{equation} \label{eq:firstswitch}
\sum_{\substack{q \sim Q \\ (q,a) = 1}} \xi_{q} \Big ( \sum_{\substack{d b \sim x \\ d \leq z \\ d b \equiv a \pmod{q} \\ (d,a) = 1}} \lambda_d - \frac{1}{\varphi(q)} \sum_{\substack{d b \sim x \\ d \leq z \\ (d,a) = 1 \\ (db,q) = 1}} \lambda_d \Big ) 
\end{equation}
with $\xi_{q}$ a sequence of complex number with $|\xi_{q}| = 1$. 

We now evaluate the first term, and express the congruence condition $d b \equiv a \pmod{q}$ as $db = a + q r$ for some $r \in \mathbb{Z}$. This shows that the first term in  \eqref{eq:firstswitch} is
\begin{equation} \label{eq:applythmtoit}
\sum_{\substack{q \sim Q \\ (q,a) = 1}} \xi_{q} \sum_{\substack{d b \sim x \\ d \leq z \\ d b = a + q r \\ (d,a) = 1}} \lambda_d = \sum_{i = 1}^{2} \sum_{\substack{(d,a) = 1 \\ d \leq z}} \lambda_d \sum_{\substack{q \sim Q, r \sim i x / (2 Q) \\ q r \sim x \\ (q,a) = 1 \\ q r \equiv - a \pmod{d}}} \xi_{q} + O(|a| x^{\varepsilon}). 
\end{equation}
Given $C > 0$ notice that the contribution of $d$ for which $|\lambda_{d}| > (\log x)^{C}$ is
$$
\ll x (\log x)^{-C} \sum_{d \leq z} \frac{\tau_k(d)^2}{d}  \ll x (\log x)^{-C + k^2}
$$
This sum over the remaining $d$'s with $|\lambda_d| \leq (\log x)^{C}$ is in a form that is amenable to our main Theorem and indeed applying Corollary \ref{cor:Main}$ii)$ (when $z \leq x^{53/105 - \varepsilon / 2}$ and $Q > x^{83/90 + \varepsilon / 2}$), Corollary \ref{cor:Main}$i)$ (when $z \leq x^{1/2 + \delta -  \varepsilon / 2}$ and $Q > x^{(71 + 66 \delta)/72 + \varepsilon / 2}$) or Corollary \ref{cor:Main}$iii)$ (when $z \leq x^{53/105 - \varepsilon / 2}$ and $Q > x^{529/630 + \varepsilon / 2}$ and $1 \leq |a| \leq x^{\varepsilon / 1000}$) we see that the sum on the right-hand side of \eqref{eq:applythmtoit} restricted to $d$ with $|\lambda_d| < (\log x)^{C}$ is equal to
$$
\sum_{\substack{d \leq z \\ (d,a) = 1 \\ |\lambda_d| \leq (\log x)^{C}}} \lambda_d \cdot \frac{1}{\varphi(d)} \sum_{\substack{q r \sim x \\ q \sim Q \\ (qr,d) = 1 \\ (q,a) = 1}} \xi_{q} + O_{A}(x (\log x)^{C - A} + |a| x^{\varepsilon}) 
$$
As before the contribution of the $d$ with $|\lambda_d| > (\log x)^{C}$ is $\ll x (\log x)^{-C + k^2}$ and so we conclude that the right-hand side of \eqref{eq:applythmtoit} is equal to
\begin{equation} \label{eq:continuing}
\sum_{\substack{d \leq z \\ (d,a) = 1}} \lambda_d \cdot \frac{1}{\varphi(d)} \sum_{\substack{q r \sim x \\ q \sim Q \\ (qr,d) = 1 \\ (q,a) = 1}} \xi_{q} + O_{A,C}(x (\log x)^{C - A} + x (\log x)^{-C + k^2} + |a| x^{\varepsilon}) 
\end{equation}
Choosing $C$ sufficiently large with respect to $k$ and $A$ sufficiently large with respect to $C$ we see that the error term is $\ll_{K} x (\log x)^{-K} + |a| x^{\varepsilon}$ for any $K > 0$. We finally notice that the main term of \eqref{eq:continuing} is equal to 
\begin{equation} \label{eq:firstterm}
\sum_{\substack{ d \leq z \\(d,a) = 1}} \lambda_{d} \cdot \frac{1}{\varphi(d)} \sum_{\substack{q \sim Q \\ (q, a d) = 1}} \xi_{q} \cdot \Big ( \frac{x}{q} \cdot \frac{\varphi(d)}{d} + O(d^{\varepsilon}) \Big ) = x \sum_{\substack{q \sim Q \\ d \leq z \\ (d, a) = 1 \\ (q, a d) = 1}} \frac{\lambda_d  \xi_{q}}{q d} + O(Q x^{\varepsilon}).
\end{equation}
We now evaluate the second term in \eqref{eq:firstswitch}. Similarly to the previous calculation, this is
\begin{align*}
\sum_{\substack{q \sim Q \\ (q,a) = 1}} \xi_{q} \cdot \frac{1}{\varphi(q)} \sum_{\substack{db \sim x \\  d \leq z \\ (db,q) = 1 \\ (d,a) = 1}} \lambda_d & = \sum_{\substack{q \sim Q \\ (q,a) = 1}} \xi_{q} \cdot \frac{1}{\varphi(q)} \sum_{\substack{d \leq z \\ (d, a q) = 1 \\ (d,a) = 1}} \lambda_d \cdot \Big ( \frac{x}{d} \cdot \frac{\varphi(q)}{q} + O(q^{\varepsilon}) \Big ) \\ & = x \sum_{\substack{q \sim Q \\ d \leq z \\ (d,a) = 1 \\ (q, a d) = 1}} \frac{\xi_{q} \lambda_d}{q d} + O(z x^{\varepsilon}).
\end{align*}
Combining the above calculation of the first and second term appearing in \eqref{eq:firstswitch} we conclude that \eqref{eq:firstswitch} is 
$$
\ll_{A} x (\log x)^{-A} + |a| x^{\varepsilon}  + z x^{\varepsilon} + Q x^{\varepsilon}
$$
for any $A > 0$. 
Since $1 \leq |a|,z,Q \leq x^{1 - 2\varepsilon}$ by assumption we conclude that the above is $\ll_{A} x (\log x)^{-A}$ for any $A > 0$ as claimed. 

\section{Proof of Corollary \ref{cor:tit}}

We pick sieve weights $\lambda_{d,z}$ such that, 
$$
\sum_{\substack{d | n \\ d \leq z}} \lambda_{d,z} = \Big ( \sum_{\substack{d | n \\ d \leq \sqrt{z}}} \mu(d) \cdot \Big ( 1 - \frac{\log d}{\log \sqrt{z}} \Big ) \Big )^2
$$
and we set $\lambda_{d,z} = 0$ for $d > z$. Observe that $|\lambda_{d,z}| \leq 3^{\omega(d)}$ where $\omega(n)$ is the number of prime factors of $n$ counted without multiplicity. Throughout $z \leq x^{1 - 2 \varepsilon}$. It is clear from the definition of $\lambda_{d,z}$ that for $n > \sqrt{z}$, 
$$
\mathbf{1}_{n \text{ is prime}} \leq \sum_{\substack{d | n}} \lambda_{d,z}.
$$
Moreover by a classical computation,
$$
\frac{1}{\varphi(q)} \sum_{\substack{n \sim x \\ (n,q) = 1}} \sum_{d | n} \lambda_{d,z} = \frac{(2 + o(1)) x}{\varphi(q) \log z}
$$
uniformly in $q \leq z^{1 - \varepsilon}$ as $x \rightarrow \infty$. 

We choose $z = x^{\frac{53}{105} - \varepsilon}$. It then follows from Corollary \ref{cor:dirichlet}$i)$ that for any $x^{83/90 + \varepsilon} \leq Q \leq x^{1 - 3 \varepsilon}$, for almost all $q \sim Q$ with at most $\ll_{A} Q (\log x)^{-A}$ exceptions we have,
$$
\sum_{\substack{p \sim x \\ p \equiv a \pmod{q}}} 1 \leq \sum_{\substack{n \sim x \\ n \equiv a \pmod{q}}} \Big ( \sum_{d | n} \lambda_{d,z} \Big ) = \frac{1}{\varphi(q)} \sum_{\substack{n \sim x \\ (n,q) = 1}} \Big ( \sum_{d | n} \lambda_{d,z} \Big ) = \frac{(2 + o(1)) x}{\varphi(q) \log z}    
$$
and this yields the claim for any $\theta > \tfrac{83}{90}$. In the regime $\theta \leq \frac{83}{90}$ we obtain a superior result by appealing to the result of Fouvry \cite{FouvryInv}, which however reduces the uniformity in $a$ to $1 \leq |a| \leq (\log x)^{C}$. Finally for the proof of the Remark, that is for $\theta = 1 - \eta$ with $0 < \eta < \frac{1}{66}$ it suffices to choose $z = x^{\frac{17}{33} - \frac{12}{11} \eta - \varepsilon }$ and appeal to Corollary \ref{cor:dirichlet}$ii)$.


\bibliographystyle{plain} \bibliography{FouvryRadzi4}

\begin{thebibliography}{10}

\bibitem{BakerHarman}
R.~C. Baker and G.~Harman.
\newblock The {B}run-{T}itchmarsh theorem on average.
\newblock In {\em Analytic number theory, {V}ol. 1 ({A}llerton {P}ark, {IL},
  1995)}, volume 138 of {\em Progr. Math.}, pages 39--103. Birkh\"auser Boston,
  Boston, MA, 1996.

\bibitem{Be-Ch}
S.~Bettin and V.~Chandee.
\newblock Trilinear forms with {K}loosterman fractions.
\newblock {\em Adv. Math.}, 328:1234--1262, 2018.

\bibitem{BCR}
S.~Bettin, V.~Chandee, and M.~Radziwi{\l}{\l}.
\newblock The mean square of the product of the {R}iemann zeta-function with
  {D}irichlet polynomials.
\newblock {\em J. Reine Angew. Math.}, 729:51--79, 2017.

\bibitem{BFI1}
E.~Bombieri, J.~B. Friedlander, and H.~Iwaniec.
\newblock Primes in arithmetic progressions to large moduli.
\newblock {\em Acta Math.}, 156(3-4):203--251, 1986.

\bibitem{B-F-I2}
E.~Bombieri, J.~B. Friedlander, and H.~Iwaniec.
\newblock Primes in arithmetic progressions to large moduli. {II}.
\newblock {\em Math. Ann.}, 277(3):361--393, 1987.

\bibitem{Polymath}
W.~Castryck, \'E. Fouvry, G.~Harcos, E.~Kowalski, P.~Michel, P.~Nelson,
  E.~Paldi, J.~Pintz, A.~V. Sutherland, T.~Tao, and X-F. Xie.
\newblock New equidistribution estimates of {Z}hang type.
\newblock {\em Algebra Number Theory}, 8(9):2067--2199, 2014.

\bibitem{D-F-I}
W.~Duke, J.~Friedlander, and H.~Iwaniec.
\newblock Bilinear forms with {K}loosterman fractions.
\newblock {\em Invent. Math.}, 128(1):23--43, 1997.

\bibitem{FoTh}
\'E Fouvry.
\newblock {\em {R}{\'e}partition des suites dans les progressions
  arithm{\'e}tiques. {R}{\'e}sultats de type {B}ombieri–{V}inogradov avec
  exposant sup{\'e}rieur {\`a} $1/2$}.
\newblock PhD thesis, Universit\'e de Bordeaux, 1981.

\bibitem{FoActaMath}
\'E. Fouvry.
\newblock Autour du th\'eor\`eme de {B}ombieri-{V}inogradov.
\newblock {\em Acta Math.}, 152(3-4):219--244, 1984.

\bibitem{FouvryInv}
\'E. Fouvry.
\newblock Th\'eor\`eme de {B}run-{T}itchmarsh: application au th\'eor\`eme de
  {F}ermat.
\newblock {\em Invent. Math.}, 79(2):383--407, 1985.

\bibitem{FoAnnENS}
\'E. Fouvry.
\newblock Autour du th\'eor\`eme de {B}ombieri-{V}inogradov. {II}.
\newblock {\em Ann. Sci. \'Ecole Norm. Sup. (4)}, 20(4):617--640, 1987.

\bibitem{FoIwMathematika}
\'E. Fouvry and H.~Iwaniec.
\newblock On a theorem of {B}ombieri-{V}inogradov type.
\newblock {\em Mathematika}, 27(2):135--152 (1981), 1980.

\bibitem{d3}
\'E. Fouvry, E.~Kowalski, and P.~Michel.
\newblock On the exponent of distribution of the ternary divisor function.
\newblock {\em Mathematika}, 61(1):121--144, 2015.

\bibitem{Fo-Ma}
{\'E}.~Fouvry and C.~Mauduit.
\newblock M\'ethodes de crible et fonctions sommes des chiffres.
\newblock {\em Acta Arith.}, 77(4):339--351, 1996.

\bibitem{Fo-Ra}
\'E. Fouvry and M.~Radziwi\l\l.
\newblock Another application of linnik dispersion method.
\newblock {\em Chebyshevski\u{\i} Sb.}

\bibitem{Opera}
J.~Friedlander and H.~Iwaniec.
\newblock {\em Opera de cribro}, volume~57 of {\em American Mathematical
  Society Colloquium Publications}.
\newblock American Mathematical Society, Providence, RI, 2010.

\bibitem{GrSh}
A.~Granville and X.~Shao.
\newblock {B}ombieri-{V}inogradov for multiplicative functions, and beyond the
  {$x^{1/2}$}-barrier.
\newblock {\em pre-print, arxiv:1703.06865}.

\bibitem{Green}
B.~Green.
\newblock A note on multiplicative functions on progressions to large moduli.
\newblock {\em Proc. Roy. Soc. Edinburgh Sect. A}, 148(1):63--77, 2018.

\bibitem{Sh}
P.~Shiu.
\newblock A {B}run-{T}itchmarsh theorem for multiplicative functions.
\newblock {\em J. Reine Angew. Math.}, 313:161--170, 1980.

\bibitem{Zhang}
Y.~Zhang.
\newblock Bounded gaps between primes.
\newblock {\em Ann. of Math. (2)}, 179(3):1121--1174, 2014.

\end{thebibliography}

 \end{document}